\documentclass[11pt,reqno]{amsart}
\usepackage[a4paper,centering,vscale=0.8,hscale=0.78]{geometry}
\usepackage[T1]{fontenc}
\usepackage[latin1]{inputenc}
\usepackage[babel]{csquotes}

\usepackage{cite}
\usepackage{amssymb}
\usepackage{amsmath}
\usepackage{amsthm}
\usepackage{latexsym}
\usepackage{graphicx}
\usepackage{mathrsfs}
\usepackage{mathrsfs}
\usepackage{bbm}
\usepackage{verbatim}

\usepackage[usenames,dvipsnames]{color} 
\usepackage[colorlinks=true, pdfstartview=FitV, linkcolor=blue, 
citecolor=red, urlcolor=blue]{hyperref}

\newtheorem{thm}{Theorem}[section]

\newtheorem{lem}[thm]{Lemma}
\newtheorem{prop}[thm]{Proposition}
\newtheorem{defin}[thm]{Definition}
\newtheorem{remark}[thm]{Remark}

\numberwithin{equation}{section}

\def\enne{\mathbb{N}}

\def\erre{\mathbb{R}}

\def\bu{{\bf u}}

\DeclareMathOperator*{\esssup}{ess\,sup}
\DeclareMathOperator*{\essinf}{ess\,inf}

\def\eps{\varepsilon}

\def\beq{\begin{equation}}
\def\eeq{\end{equation}}

\def\to{\rightarrow}
\def\wto{\rightharpoonup}
\def\wstarto{\stackrel{*}{\rightharpoonup}}
\def\embed{\hookrightarrow}

\def\norm #1{\left\|#1\right\|}

\def\ip #1#2{\left<#1,#2\right>}

\begin{document}
\title[Nonlinear evolution equations via modular spaces]
{An extended variational theory for\\nonlinear evolution equations via modular spaces}

\author{Alexander Menovschikov}
\address[Alexander Menovschikov]{Department of Mathematics, University of Hradec Kr\'alov\'e, 
Rokitansk\'eho 62, 500 03 Hradec Kr\'alov\'e, Czech Republic, \&
Faculty of Economics, University of South Bohemia, Studentsk\'a 13, 370 05 \v{C}esk\'e Bud\v{e}jovice, Czech Republic}
\email{alexander.menovschikov@uhk.cz}
\urladdr{}

\author{Anastasia Molchanova}
\address[Anastasia Molchanova]{Institute of Analysis and Scientific Computing TU Vienna,
Wiedner Hauptstra{\ss}e 8-10, 1040 Wien, Austria, \&
Faculty of Mathematics, University of Vienna, Oskar-Morgenstern-Platz 1, 1090 Wien, Austria}
\email{anastasia.molchanova@univie.ac.at}
\urladdr{https://www.mat.univie.ac.at/$\sim$molchanova/}

\author{Luca Scarpa}
\address[Luca Scarpa]{Department of Mathematics, Politecnico di Milano, 
Via E.~Bonardi 9, 20133 Milano, Italy, \&
Faculty of Mathematics, University of Vienna, 
Oskar-Morgenstern-Platz 1, 1090 Wien, Austria}
\email{luca.scarpa@polimi.it}
\urladdr{http://www.mat.univie.ac.at/$\sim$scarpa}



\begin{abstract}
  We propose an extension of the classical variational theory of evolution equations
  that accounts for dynamics also in possibly non-reflexive and non-separable spaces.
  The pivoting point is to establish a novel variational structure, 
  based on abstract modular spaces associated to a given convex function.
  Firstly, we show that the new variational triple is suited for framing the evolution, 
  in the sense that a novel duality paring can be introduced and a 
  generalised computational chain rule holds. Secondly, we prove well-posedness 
  in an extended variational sense for evolution equations, 
  without relying on any reflexivity assumption
  and any polynomial requirement on the nonlinearity.
  Finally, we discuss 
  several important applications that can be addressed in this framework:
  these cover, but are not limited to, equations in Musielak--Orlicz--Sobolev spaces, 
  such as
  variable exponent, Orlicz, weighted Lebesgue, and double-phase spaces.
\end{abstract}

\maketitle


\section{Introduction}
\setcounter{equation}{0}
\label{sec:intro}
In this paper we deal with evolution equations on a Hilbert space $H$
in the form
\beq
  \label{eq0}
  \partial_t u + A(u)\ni f\,, \qquad u(0)=u_0\,,
\eeq
where $A:=\partial\varphi$ is the subdifferential of
a proper, convex, and lower semicontinuous function $\varphi\colon H\to(-\infty,+\infty]$.
Moreover, $f\colon (0,T)\to H$ is a prescribed forcing term, with $T>0$ being the final reference time,
and the initial datum $u_0$ is given in $H$.

From the mathematical perspective, nonlinear evolution equations in Hilbert or Banach spaces
have been extensively investigated in the last decades. 
Starting from the pioneering literature of the 70s,
for which we refer the reader to
Barbu \cite[Ch.~4]{barbu-monot},
their study represents nowadays one of the most flourishing fields of modern 
mathematical analysis, with applications
ranging from partial differential equations
to functional analysis. In this regard, 
we point out the contributions 
\cite{diben-show, colli-visin, colli, sch-seg-stef, stef}
on doubly nonlinear evolution equations, 
and \cite{ak-stef, ak-stef2, miel-stef} on variational principles,
as well as the references therein.
Well-posedness for equations in the form \eqref{eq0} has been tackled in 
several frameworks, and various concepts of solutions have been proposed.
The specific regularity of the solutions strongly depends on the assumptions
on the initial datum $u_0$ and the forcing term~$f$.

The most classical (and the strongest) assumption on the initial datum is that 
$u_0\in D(A)$, where $D(A)$
stands for the effective domain of the maximal monotone operator $A$ on $H$.
In this framework, existence of strong solutions for \eqref{eq0} has been 
thoroughly studied in the Hilbert case 
in relation to the nonlinear extension of the celebrated Hille--Yosida theory:
this was 
first accomplished by K\={o}mura \cite{kom3}
and then the result was well reviewed in the monograph 
by Brezis \cite{brezis},
not necessarily requiring $A$ to be cyclically monotone.
More in detail, the corresponding homogeneous equation (i.e.~with $f=0$)
admits a unique strong solution $u$, in the sense that $u\in W^{1,\infty}(0,T; H)$
and the differential inclusion \eqref{eq0} is satisfied almost everywhere on $(0,T)$.
In the nonhomogeneous case, existence of strong solutions in $W^{1,\infty}(0,T; H)$
is ensured if the forcing term satisfies at least $f\in BV(0,T; H)$.
In the special case of subdifferential operators $A=\partial\varphi$,
existence of strong solutions in $W^{1,p}(0,T; H)$, for $p\in[2,+\infty)$, is ensured 
if $f\in L^p(0,T; H)$.
Here, the proof can be based on a regularisation and 
passage-to-the-limit procedure based on a Yosida-type approximation
of the operator $A$.

However, in several situations the assumption that the 
initial datum belongs to the domain of $A$ is too restrictive.
In this direction, weaker concepts of solution have been proposed 
in order to give appropriate meaning to equation \eqref{eq0} in more general frameworks.
In this direction, the initial datum can be required to satisfy only $u_0\in\overline{D(A)}$, 
the closure in $H$ of the domain of $A$, and the forcing 
term is only taken as $f\in L^1(0,T; H)$.
Under these condition, in Brezis \cite{brezis}
existence and uniqueness of a weak solution
$u\in C^0([0,T]; H)$ is proved, 
in the sense that there exist some sequences of data 
$(u_{0,n}, f_n)_n\subset D(A)\times L^1(0,T; H)$
and a sequence of respective strong solutions $(u_n)_n\in W^{1,1}(0,T; H)$,
satisfying \eqref{eq0} for every $n\in\enne$, and such that, as $n\to\infty$,
$u_{0,n}\to u_0$ in~$H$, $f_n\to f$ in $L^1(0,T; H)$,
and $u_n\to u$ in $C^0([0,T]; H)$. 
This concept of weak solution is certainly satisfactory,
as it allows to give sense to the equation in more general 
settings, and it generalises the notion of strong solution.
Furthermore, whenever $A$ is cyclically monotone 
with $A=\partial\varphi$, every weak solution is also a strong 
solution as soon as $u_0\in D(\partial\varphi)$ and $f\in L^2(0,T; H)$, for example.

Alternative concepts of weak solutions 
have been proposed and studied, also for evolutions in Banach spaces,
especially relying on the notion of mild and integral solutions.
In this direction, the reader can refer to the classical pioneering works by 
B\'enilan \cite{ben}, B\'enilan \& Brezis \cite{ben-brezis},
Crandall \& Evans \cite{cran-evans}, Crandall \& Liggett \cite{cran-ligg},
Crandall \& Pazy \cite{cran-pazy, cran-pazy2}, Kato \cite{kato, kato2},
and Komura \cite{kom, kom2}. 

One of the main drawbacks of the above-mentioned notion of weak solution is that 
the differential inclusion $f- \partial_t u \in A(u)$
is not explicitly specified. More precisely, 
for weak solutions $u$ the meaning of the differential inclusion 
is somehow passed by: although 
$u$ is obtained as limit of suitable approximations
(strong solutions, time discretisation, etc.), each one solving the differential 
inclusion almost everywhere, it is not granted that $u$
itself satisfies \eqref{eq0} in some sense on $(0,T)$.
This issue has been successfully overcome by proposing a further concept 
of weak solution for the problem, for which the differential 
inclusion could be made explicit in some appropriate sense: 
it is the case of the well-celebrated 
variational theory introduced by Lions \cite{lions}.
The main idea behind this is to introduce, apart from the 
given Hilbert space $H$, a further Banach space $V$, which is 
assumed to be separable, reflexive, and continuously and densely embedded in~$H$.
If one identifies~$H$ with its dual $H^*$, then it is possible to work on the variational triplet 
$V\embed H \embed V^*$.
The maximal monotone operator $A$ is then considered as an operator $A\colon V\to 2^{V^*}$
satisfying some natural coercivity and boundedness conditions in the form
\beq\label{pol_coerc}
  \ip{y}{\varphi}_{V^*,V}\geq c\norm{\varphi}^p\,, \quad
  \norm{y}_{V^*}^{\frac{p}{p-1}}\leq C\left(1+\norm{\varphi}^p_V\right)
  \qquad\forall\,\varphi\in V\,,\quad\forall\,y\in A(\varphi)\,,
\eeq
where $p>1$ is a given constant.
In this framework, for every initial datum $u_0\in H$ and forcing term $f\in L^1(0,T; H)$
there exists a unique variational solution $u$, in the sense that 
$u\in W^{1,\frac{p}{p-1}}(0,T; V^*)\cap L^p(0,T; V)$ satisfies the differential inclusion
\eqref{eq0} as an equality in $V^*$, almost everywhere on $(0,T)$. 
Let us stress that variational solutions are effective in 
rendering the differential inclusion explicitly for $u$,
and not only as limit of suitable approximations: indeed, 
one is actually able to prove that the inclusion $f-\partial_tu \in A(u)$
holds in $V^*$, almost everywhere on $(0,T)$.

The main downside of the classical variational theory
is that the operator $A$ needs to possess relatively nice 
polynomial behaviour. While on the one hand
this is certainly enough to cover several classes of evolution equations,
such as reaction-diffusion and $p$-Laplace equations, 
on the other hand numerous important applications are left out.
This happens in particular when the operator $A$ fails 
to satisfy the coercivity-boundedness polynomial conditions \eqref{pol_coerc}
for a certain exponent $p$. Quite common $A$ is actually needed to be coercive and bounded in some spaces, but the respective 
exponents are different: this is very classical, 
for example, for $p(x)$-Laplace and double-phase equations, as well as 
equations in weighted Sobolev spaces. Alternatively,
it may happen that $A$ is coercive and bounded as required, 
but the natural space $V$ associated to it is not reflexive or not separable:
this is the case, among many others, of
reaction-diffusion equations with singular (i.e.~superpolynomial) potentials.
In all these pathological scenarios, if the initial datum only satisfies $u_0\in H$
an appropriate concept of variational solution is not known so far:
one is only able to obtain strong solutions to \eqref{eq0}
by forcing the initial datum to belong to $D(A)$.

These issues naturally call for an extension of the classical variational theory,
in order to establish, in some appropriate variational sense, weak well-posedness of 
evolution equations in the form \eqref{eq0} with general initial datum $u_0\in H$, 
also when assumption \eqref{pol_coerc} is not satisfied.
This is the main focus of the present paper.

The idea is to abandon the introduction of the Banach space $V$,
and to work instead in the modular spaces naturally associated to $\varphi$.
Indeed, one can define the small and large modular spaces as, respectively,
\begin{align*}
	L_\varphi := \{v \in H :\; \exists \,\alpha >0 :\quad \varphi(\alpha v) < +\infty\}\,,\\
	E_\varphi := \{v \in H :\; \forall\,\alpha >0 :\quad \varphi(\alpha v) < +\infty\}\,.
\end{align*}
As $E_\varphi$ is generally dense in $H$ 
(details are given in Section~\ref{sec:general_setting} below),
by identifying $H$ with its dual one has the natural variational structure 
\[
  E_\varphi \embed L_\varphi \embed H \embed E_\varphi^*\,.
\]
In the classical variational theory, one is implicitly assuming 
that $E_\varphi=L_\varphi=V$ is separable and reflexive: this is satisfied 
only in very specific situations, for example when 
$\varphi$ satisfy some suitable $\Delta_2$ and $\nabla_2$ conditions
(details are given in Section~\ref{sec:applications} below).
In general, however, the spaces $E_\varphi$ and $L_\varphi$ are different
and not necessarily reflexive. The pivoting idea of the entire work is 
to work in the variational triplet 
\[
  L_\varphi \embed H \embed E_\varphi^*\,.
\]
Beyond the reflexivity issue, 
a further problem is that $E_\varphi^*$ is {\it not} the dual of $L_\varphi$, hence no
duality pairing is in principle defined between $E_\varphi^*$ and $L_\varphi$.
The preliminary step is then to show that, nonetheless, it is possible to
define a new duality $[\cdot, \cdot]$ between $E_\varphi^*$ and $L_\varphi$
generalising the scalar product of $H$. This guarantees indeed that the 
triple $(L_\varphi, H, E_\varphi^*)$, despite being unconventional in this sense, 
is suited for framing the evolution problem in a variational way:
the solution $u$ is expected to be $L_\varphi$-valued, and the differential 
inclusion \eqref{eq0} will be intended in $E_\varphi^*$.

The first main result of the paper is a fundamental computational tool
collected in Theorem~\ref{thm:chain} below, establishing 
that the novel variational triplet $(L_\varphi, H, E_\varphi^*)$ endowed 
with the novel duality pairing $[\cdot, \cdot]$ satisfies 
the well-known ``chain rule'' formula for the square of the $H$-norm.
This is highly nontrivial, since the spaces $L_\varphi$ and $E_\varphi^*$ are not 
reflexive and separable in general, hence the classical results do not apply.
In particular, the non-separability of the spaces in play forces to 
introduce a different notion of measurability for vector-valued functions, 
as the classical strong measurability in the Bochner sense is too restrictive in this framework.
The chain rule is proved using an elliptic-in-time regularisation 
by convolution and passage to the limit, where a key role is played
by an abstract version of the Jensen inequality proved by Haase \cite{haase}.

The second main result of the paper is contained in Theorems~\ref{thm}--\ref{thm2}
and establishes the variational well-posedness of equation \eqref{eq0}
is the new variational setting $(L_\varphi, H, E_\varphi^*)$.
The structure of the prof is based on a Yosida-type approximation on $\partial\varphi$
and passage to the limit. Let us stress that due to the lack of reflexivity
several compactness issues arise, especially in the direction of identifying 
the nonlinearity $\partial\varphi$ at the limit.

The main novelty of this paper is to provide an extended variational structure 
that allows to frame also singular evolution equations in possibly non-reflexive spaces.
This is fundamental as it provides a unifying variational framework 
for a wide variety of problems,
such as equations in Musielak--Sobolev--Orlicz spaces,
that so far have been studied independently by hand.
In this regard, Musielak--Orlicz spaces and their special cases 
are receiving much attention at the present time.
A survey of nonlinear PDEs in Musielak--Orlicz spaces
is presented in Chlebicka \cite{Chl2018}, 
especially in the cases of variable exponent, 
Orlicz, weighted Lebesgue, and double-phase spaces.
Further contributions on Musielak--Orlicz--Sobolev spaces are
also given in the recent articles 
\cite{Bul2019, Chl20182, Fis2020} and \cite{Har2016, Ort2018, Yan2019}.
Such spaces were introduced in the late 50s in the works of Musielak and Orlicz
\cite{MusOrl1959-1, MusOrl1959-2}.
The theory has then been viewed as a special case of a more general approach,
based on modular spaces: the main idea is to consider a convex functional (a so-called \textit{modular})
on a vector space instead of the usual integrals of convex functions
involved in Musielak--Orlicz spaces. 
This approach is widely employed already in several fields, such as 
approximation and interpolation theory (e.g.~\cite{Koz1989, Hud1980, youss, Ahmi2018}) 
and in operator theory (e.g.~\cite{Kha1990, Dom2020, Pak2013, Raj2012}). 
Also, metric space theory on such spaces is developed (see Chistyakov~\cite{Chist2015}). 
Recent results concerning existence of solutions 
to parabolic equations in Orlicz spaces have been obtained in
\cite{el-mes1, el-mes2, ben-with-zimm, chl, Chl20182, gw-zimm, zhan-zhou}.
For a highly detailed presentation of the existing literature 
we refer specifically to \cite[\S~3]{Chl2018}.

The importance of the extended variational approach introduced in this work
is extremely evident in all those scenarios where existence of strong 
solutions is out of reach, due to some specific pathological structure 
of the problems themselves.
In this regard, a special mention goes to nonlinear evolution equations
with random forcing. Indeed, in the stochastic setting existence of analytically 
strong solutions may be severely problematic if the nonlinearity $A$ is too
singular, due to the presence of extra second order contributions in the energy balance
(see for example Gess \cite{gess} for the case of sub-homogenous potential).
Consequently, forcing the initial datum to belong to $D(A)$ does not help in this case,
and it is fundamental to have at hand a valid well-posedness theory 
in a variational sense, in order to give appropriate meaning to possibly singular 
evolution equations with general $u_0\in H$. For these reasons, 
we believe that the current work represents also a preliminary step 
in the direction of building a generalised variational theory for stochastic 
evolution equations.
The variational theory for SPDEs was 
originally introduced by Pardoux \cite{Pard} and 
Krylov \& Rozovski{\u\i} \cite{KR-spde}
(see also \cite{LiuRo} and the references therein)
under the classical reflexivity--separability conditions on the space $V$
and under the polynomial requirements \eqref{pol_coerc} on the nonlinearity.
Some first contributions in the spirit of Orlicz spaces have been 
given so far only in very special cases, 
namely in 
Barbu \& Da Prato \& R\"{o}ckner \cite[Ch.~4]{barbu-daprato-rock-book} for the stochastic 
porous media equation, in 
\cite{barbu-semilin, mar-scar-diss, mar-scar-ref, orr-scar}
for semilinear stochastic equations, in \cite{mar-scar-div, mar-scar-note} for 
stochastic divergence-form equations, and in \cite{scar-SCH, scar-SVCH}
for the stochastic Cahn--Hilliard equation.
Nonetheless, a general extended variational theory for stochastic 
evolution equations taking into account possibly non-reflexive spaces
and non-polynomial nonlinearities is missing: in this direction,
the present work represents a valuable candidate 
for obtaining an extension to the stochastic case, which is itself currently in preparation.

Finally, let us briefly present the structure of the paper.
In Section~\ref{sec:modular} we collect some preliminary general results 
on modular spaces, while in Section~\ref{sec:general_setting} we 
introduced the novel variational setting and state our main results.
Section~\ref{sec:chain} contains then the proof of the generalised chain rule,
and Section~\ref{sec:proof} is focused on the proof of well-posedness.
Eventually, in Section~\ref{sec:applications} we collect several important 
applications that can be treated in this framework.


\section{Preliminaries on modular spaces}
\label{sec:modular}
In this section, we recall the main definitions and properties
concerning the theory of modular spaces, and 
we prove some preliminary abstract results that 
will be useful in the sequel.
For the details on the theory of modular spaces, we refer the reader to Musielak \cite{mus}.

Let $X$ be a real Banach space with dual $X^*$.
The norm in $X$ and the duality between $X^*$
and $X$ will be denoted by the symbols $\norm{\cdot}_X$
and $\ip{\cdot}{\cdot}_{X^*,X}$, respectively.

\begin{defin}
	A convex semi-modular on $X$
	is a convex functional $\varphi \colon X \to [0, \infty]$ 
	satisfying the following conditions:
	\begin{itemize}
		\item $\varphi(0)=0$,
		\item if $\varphi(\alpha x)=0$ for all $\alpha>0$, then $x=0$;
		\item $\varphi(-x) = \varphi(x)$ for all $x \in X$.
	\end{itemize}
	If also $\varphi(x)=0$ iff $x=0$, then $\varphi$ is called convex modular.
\end{defin}

In this section, $\varphi$ is a lower semicontinuous convex semi-modular on $X$.
It is possible to naturally associate to $\varphi$ the modular spaces
\begin{align}
	\label{L_phi}
	L_\varphi := \{x \in X :\; \exists \,\alpha >0 :\quad \varphi(\alpha x) < +\infty\}\,,\\
	\label{E_phi}
	E_\varphi := \{x \in X :\; \forall\,\alpha >0 :\quad \varphi(\alpha x) < +\infty\}\,.
\end{align}
It is not difficult to check that $E_\varphi$ and $L_\varphi$ are real linear spaces, with 
\[
  E_\varphi \subset L_\varphi \subset X\,.
\]
Furthermore, we set
\begin{equation}\label{lux}
	\|x\|_{\varphi} := \inf \{\lambda > 0:\quad \varphi(x/\lambda) \leq 1 \}\,, \qquad x\in L_\varphi\,.
\end{equation} 
It is well-known (see for example \cite[Chapter I]{mus}) that $\norm{\cdot}_\varphi$
defines a norm on $E_\varphi$ and $L_\varphi$, called the \textit{Luxemburg} norm,
so that $(E_\varphi, \norm{\cdot}_\varphi)$ and $(L_\varphi, \norm{\cdot}_\varphi)$
are linear normed spaces. From the definition of 
$\norm{\cdot}_\varphi$, the properties collected in the following Lemma are well-known: we refer 
for example to Musielak~\cite[Thms.~1.5--1.6, Lem.~2.4]{mus}.
\begin{lem}\label{lem:prop}
  The following properties hold:
  \begin{enumerate}
  \item for every $x_1,x_2\in L_\varphi$, 
  if $\varphi(\alpha x_1)\leq\varphi(\alpha x_2)$ for all $\alpha>0$,
  then $\norm{x_1}_{\varphi}\leq\norm{x_2}_{\varphi}$;
  \item for every $x\in L_\varphi$ with $\norm{x}_{\varphi}<1$, 
  it holds $\varphi(x)\leq\norm{x}_{\varphi}$;
  \item for every $x\in L_\varphi$ with $\norm{x}_{\varphi}>1$, 
  it holds $\varphi(x)\geq\norm{x}_{\varphi}$;
  \item for every sequence $(x_n)_n\subset L_\varphi$ and $x\in L_\varphi$, it holds
  \[
  \lim_{n\to\infty}\norm{x_n-x}_\varphi=0 \qquad\text{iff}\qquad
  \lim_{n\to\infty}\varphi(\alpha(x_n-x))=0 \quad\forall\,\alpha>0\,;
  \]
  \item for every sequence $(x_n)_n\subset L_\varphi$, it holds
  \[
  \lim_{n,k\to\infty}\norm{x_n-x_k}_\varphi=0 \qquad\text{iff}\qquad
  \lim_{n,k\to\infty}\varphi(\alpha(x_n-x_k))=0 \quad\forall\,\alpha>0\,.
  \]
  \end{enumerate}
\end{lem}

Since we are interested in applications to evolutionary PDEs in modular spaces, 
we look now for sufficient conditions on $\varphi$ ensuring that 
$(E_\varphi, \norm{\cdot}_\varphi)$ and $(L_\varphi, \norm{\cdot}_\varphi)$ are actually
Banach spaces. In the direction, we have the following result.
\begin{prop}
  \label{prop:banach}
  Suppose that there exists a strictly increasing function 
  $\rho:[0,+\infty)\to[0,+\infty)$ with  $\rho(0)=0$ such that 
  \beq\label{ip_varphi}
  \varphi(x)\geq \rho(\norm{x}_X) \qquad\forall\,x\in X\,.
  \eeq
  Then, $(E_\varphi, \norm{\cdot}_\varphi)$ and $(L_\varphi, \norm{\cdot}_\varphi)$ are
  Banach spaces, and the following inclusions are continuous:
  \[
  E_\varphi \embed L_\varphi \embed X\,.
  \]
\end{prop}
\begin{proof}
  {\sc Step 1}.
  Firstly, we prove that $(E_\varphi, \norm{\cdot}_\varphi)$ is a Banach space.
  Let $(x_n)_n\subset E_\varphi$ be a Cauchy sequence: then, there exists 
  an index $\bar m\in\enne$ such that, for every $n$, $k\geq\bar m$ we have
  $\norm{x_n-x_k}_\varphi<1$. Consequently, by Lemma~\ref{lem:prop} (2)
  and the assumption \eqref{ip_varphi} 
  we have
  \[
  \rho(\norm{x_n-x_k}_X)
  \leq \varphi(x_n-x_k)\leq \norm{x_n-x_k}_\varphi \qquad\forall\,n,k\geq\bar m\,.
  \]
  In particular, it follows that $(x_n)_n$ is a Cauchy sequence in $X$. By completeness of $X$,
  there exists $x\in X$ such that $x_n\to x$ in $X$, hence also $\alpha x_n\to \alpha x$
  in $X$ for all $\alpha>0$. 
  Let now $\alpha>0$ be arbitrary: since $(x_n)_n$ is Cauchy in $E_\varphi\subset L_\varphi$, 
  by Lemma~\ref{lem:prop} (5) we have that
  \[
  \lim_{n,k\to\infty}\varphi(\alpha(x_n-x_k))=0\,,
  \]
  so that there exists $\bar m_\alpha\in\enne$ such that 
  \[
  \varphi(\alpha(x_n-x_k))\leq 1 \qquad\forall\,n,k\geq\bar m_\alpha\,.
  \]
  By lower semicontinuity and convexity of $\varphi$ in $X$,
  since $x_{\bar m_\alpha}\in E_\varphi$ we deduce that
  \begin{align*}
  \varphi\left(\frac\alpha2 x\right)&\leq\liminf_{n\to\infty}\varphi\left(\frac\alpha2 x_n\right)\\
  & = \liminf_{n\to\infty}\varphi\left(\frac12\alpha(x_n-x_{\bar m_\alpha}) 
  + \frac12\alpha x_{\bar m_\alpha}\right)\\
  &\leq \frac12\liminf_{n\to\infty}\varphi(\alpha(x_n-x_{\bar m_\alpha})) 
  + \frac12\varphi(\alpha x_{\bar m_\alpha}) \\
  &\leq 1 + \frac12\varphi(\alpha x_{\bar m_\alpha})<+\infty\,.
  \end{align*}
  Hence, $\varphi(\frac\alpha2 x)<+\infty$ for all $\alpha>0$, from which $x\in E_\varphi$.
  Moreover, again by lower semicontinuity of $\varphi$, for every $\alpha>0$ we have
  \[
  \varphi(\alpha(x_n-x))\leq\liminf_{k\to\infty}\varphi(\alpha(x_n-x_k))\,, \qquad\forall\,n\in\enne\,,
  \]
  which yields
  \[
  \limsup_{n\to\infty}\varphi(\alpha(x_n-x))\leq\limsup_{n,k\to\infty}\varphi(\alpha(x_n-x_k))=0\,.
  \]
  Hence $x_n\to x$ in $E_\varphi$ by Lemma~\ref{lem:prop} (4). This shows that 
  $(E_\varphi, \norm{\cdot}_\varphi)$ is a Banach space.
  \smallskip\\ 
  {\sc Step 2}. We prove now that also $(L_\varphi, \norm{\cdot}_\varphi)$ is a Banach space.
  Let $(x_n)_n\subset L_\varphi$ be a Cauchy sequence. Arguing exactly as in
  {\sc Step 1} we deduce that there exists $x\in X$ such that $x_n\to x$ in~$X$:
  we have to show that $x\in L_\varphi$ and $x_n\to x$ in~$L_\varphi$.
  To this end, note that since $(x_n)_n$ is Cauchy in~$L_\varphi$, 
  by the triangular inequality the real sequence $(\norm{x_n}_\varphi)_n$
  is Cauchy in $\erre$, so that there exists $\lambda\geq0$ such that 
  $\lambda_n:=\norm{x_n}_\varphi\to\lambda$ as $n\to\infty$. 
  Now, if $\lambda=0$, then $x=0\in L_\varphi$ and $x_n\to0$ in~$L_\varphi$,
  so the conclusion is trivial.
  Let us suppose that $\lambda>0$: then
  $\lambda_n>\lambda/2$ for every $n$ sufficiently large. It follows that
  \[
  \limsup_{n\to\infty}\norm{\frac{x_n}{\lambda_n} - \frac x{\lambda}}_X
  \leq\frac2{\lambda}\limsup_{n\to\infty}\norm{x_n-x}_X + 
  \norm{x}_X\limsup_{n\to\infty}\left|\frac1{\lambda_n} - \frac1{\lambda}\right| =0\,,
  \]
  hence $x_n/\lambda_n \to x/\lambda$ in $X$. By lower semicontinuity of $\varphi$
  and definition of $\lambda_n$
  we have then
  \[
  \varphi(x/\lambda)\leq\liminf_{n\to\infty}\varphi(x_n/\lambda_n)\leq 1\,,
  \]
  which implies that $x\in L_\varphi$ and $\norm{x}_\varphi\leq \lambda$.
  Finally, let $\alpha>0$ be arbitrary: then, 
  again by lower semicontinuity of $\varphi$ we have
  \[
  \varphi(\alpha(x_n-x))\leq\liminf_{k\to\infty}\varphi(\alpha(x_n-x_k))\,, \qquad\forall\,n\in\enne\,,
  \]
  which yields
  \[
  \limsup_{n\to\infty}\varphi(\alpha(x_n-x))\leq\limsup_{n,k\to\infty}\varphi(\alpha(x_n-x_k))=0\,.
  \]
  Since $\alpha>0$ is arbitrary, we have
  $x_n\to x$ in $L_\varphi$ by Lemma~\ref{lem:prop} (4). This shows that 
  $(L_\varphi, \norm{\cdot}_\varphi)$ is a Banach space.
  \smallskip\\
  {\sc Step 3}. We prove here the continuous inclusions $E_\varphi\embed L_\varphi\embed X$.
  We already know that $E_\varphi\subset L_\varphi\subset X$ as inclusions of sets.
  Moreover, the continuous inclusion $E_\varphi\embed L_\varphi$ is trivial since
  \[
  \norm{x}_{E_\varphi}=\norm{x}_{L_\varphi}=\norm{x}_\varphi \qquad\forall\,x\in E_\varphi\,.
  \]
  Let now $x\in L_\varphi$ and $\lambda>0$ such that $\varphi(x/\lambda)\leq1$. Then
  by assumption on $\varphi$ we have
  \[
  \rho(\norm{x/\lambda}_X)\leq \varphi(x/\lambda)\leq 1\,,
  \]
  from which $\norm{x}_X\leq\rho^{-1}(1)\lambda$,
  where $\rho^{-1}$ denotes the generalised inverse of $\rho$.
   By arbitrariness of $\lambda$ and 
  definition of $\norm{\cdot}_\varphi$, we have then
  \[
  \norm{x}_X\leq\rho^{-1}(1)
  \norm{x}_\varphi \qquad\forall\,x\in L_\varphi\,,
  \]
  as required.
\end{proof}

Proposition~\ref{prop:banach} ensures then that for
a lower semicontinuous convex semi-modular $\varphi$ satisfying 
condition \eqref{ip_varphi}, the respective 
modular spaces $(E_\varphi, \norm{\cdot}_\varphi)$ and 
$(L_\varphi, \norm{\cdot}_\varphi)$ are actually
Banach spaces. The next main issue 
in studying the variational structure of $E_\varphi$ and $L_\varphi$
concerns density properties.
In this direction, although 
the inclusion $E_\varphi\embed X$ and $L_\varphi\embed X$
may be dense
in the majority of applications,
let us stress that the inclusion $E_\varphi\embed L_\varphi$
is not necessarily dense in general: see Section~\ref{sec:applications} for details.
This calls for the introduction of a weaker concept of convergence in $L_\varphi$,
namely the so-called \textit{modular} convergence: see Musielak \cite{mus}.
\begin{defin}[Modular convergence]
	A sequence $(x_n)_n\subset L_\varphi$ 
	is called modular convergent to $x \in L_\varphi$
	if there exist an $\alpha > 0$ such that $\varphi(\alpha(x_n - x)) \to 0$ as $n \to \infty$.
\end{defin}
\noindent Thanks to Lemma~\ref{lem:prop}, it is clear that
modular convergence is weaker than the norm convergence. 
Actually, the former is strictly weaker than the latter, and they
are equivalent if and only if $\varphi(x_n) \to 0$ implies $\varphi(2x_n) \to 0$,
for every sequence $(x_n)_n\subset L_\varphi$.

We conclude the preliminary section with an overview on 
the duality properties of $E_\varphi$ and $L_\varphi$.
These are indeed crucial in order to 
build a suitable variational framework for PDEs in modular spaces.
\begin{defin}
	The convex conjugate $\varphi^* \colon X^* \to [0,\infty]$ 
	of $\varphi \colon X \to [0, \infty]$ is defined as
	$$
	\varphi^*(y) := \sup_{x \in X} \{\ip{y}{x}_{X^*,X} - \varphi(x)\}\,, \qquad y\in X\,.
	$$
\end{defin}

\begin{lem}\label{lem:dual}
  If $\varphi$ is a lower semicontinuous convex semi-modular on $X$
  and
  $L_\varphi\embed X$ densely, then
  $\varphi^*\colon X^*\to[0,+\infty]$ is a lower semicontinuous convex semi-modular on $X^*$.
\end{lem}
\begin{proof}
  We know that $\varphi^*$ is lower semicontinuous, proper, and convex.
  It is also immediate to check using the definition that 
  $\varphi^*(0)=0$ and $\varphi^*(-y)=\varphi^*(y)$
  for all $y\in X^*$. 
  Moreover, let $y\in X^*$ be such that $\varphi^*(\alpha y)=0$
  for every $\alpha>0$. Take now an arbitrary $x\in L_\varphi$
  and choose $\eta>0$ such that $\varphi(\eta x)<+\infty$.
  Then, by the Young inequality and the symmetry of $\varphi$,
  for all $\alpha>0$ we have
  \[
  \pm\eta\alpha\ip{y}{x}_{X^*, X}\leq \varphi(\pm\eta x) + \varphi^*(\alpha y) = \varphi(\eta x)\,,
  \]
  yielding 
  \[
  |\ip{y}{x}_{X^*, X}|\leq\frac{\varphi(\eta x)}{\eta\alpha} \qquad\forall\,\alpha>0\,.
  \]
  Since $\varphi(\eta x)<+\infty$, letting $\alpha\to+\infty$ we have
  \[
  \ip{y}{x}_{X^*, X}=0 \qquad\forall\,x\in L_\varphi\,,
  \]
  from which $y=0$ in $X^*$ be density of $L_\varphi$ in $X$.
  Hence, $\varphi^*$ is a semi-modular on $X^*$.
\end{proof}
\noindent This lemma ensures that $\varphi^*$
is always a (lower semicontinuous convex) semi-modular whenever 
so is $\varphi$ and $L_\varphi$ is dense in $X$.
However, note that 
even if we additionally require that $\varphi$ is a modular,
it is not necessarily true that $\varphi^*$
is a modular as well.
Without additional assumptions on $\varphi$, this is actually false
(as it happens for example for $\varphi(x)=\norm{x}_X$, $x\in X$).

The last results that we present here concern the duality properties of 
the restriction of $\varphi$ to $E_\varphi$. These will be fundamental 
in the paper.
\begin{lem}\label{lem:dual2}
  Let $\varphi$ be a lower semicontinuous convex semi-modular on $X$
  satisfying condition~\eqref{ip_varphi}. Then, the restriction 
  \[
  \bar\varphi\colon E_\varphi\to[0,+\infty)\,, \qquad\bar\varphi:=\varphi_{|E_\varphi}\,,
  \]
  is a lower semicontinuous convex semi-modular on $E_\varphi$, and
  its convex conjugate 
  \[
  \bar\varphi^*\colon E_\varphi^*\to[0,+\infty]\,, \qquad
  \bar\varphi^*(y):=
  \sup_{x \in E_\varphi} \{\ip{y}{x}_{E_\varphi^*, E_\varphi} - \varphi(x)\}\,, \qquad y\in E_\varphi^*\,,
  \]
  is a lower semicontinuous convex semi-modular on $E_\varphi^*$.
\end{lem}
\begin{proof}
  It is clear that $\bar\varphi$ is proper, convex and lower semicontinuous,
  since $E_\varphi\embed X$ continuously by Proposition~\ref{prop:banach}
  and $\varphi$ is lower semicontinuous on $X$. It is also trivial 
  that $\bar\varphi(0)=0$ and that $\bar\varphi(-x)=\bar\varphi(x)$
  for all $x\in E_\varphi$. Moreover, if $x\in E_\varphi$
  satisfies $\bar\varphi(\alpha x)=0$ for all $\alpha>0$, then clearly 
  $\varphi(\alpha x)=0$ for all $\alpha>0$, hence $x=0$ since
  $\varphi$ is a semi-modular. This shows that $\bar\varphi$ is 
  a lower semicontinuous convex semi-modular on $E_\varphi$.
  Consequently, since $E_{\bar\varphi}=L_{\bar\varphi}=E_\varphi$
  (so in particular $L_{\bar\varphi}$ is trivially dense in $E_{\varphi}$),
  by Lemma~\ref{lem:dual} with the choice $X=E_\varphi$ we have that 
  $\bar\varphi^*$ is a lower semicontinuous convex semi-modular on $E_\varphi^*$.
\end{proof}

\noindent Lemma~\ref{lem:dual2} ensures then that the modular spaces 
\begin{align*}
	L_{\bar\varphi^*} := \{y \in E_{\varphi}^* :\; 
	\exists \,\alpha >0 :\quad \bar\varphi^*(\alpha y) < +\infty\}\,,\\
	E_{\bar\varphi^*} := \{y \in E_\varphi^* :\; 
	\forall\,\alpha >0 :\quad \bar\varphi^*(\alpha y) < +\infty\}\,,
 \end{align*}
endowed with the norm
\[
  \norm{y}_{\bar\varphi^*}:=\inf\left\{\lambda>0:\quad\bar\varphi^*(y/\lambda)\leq1\right\}\,,
  \qquad y\in L_{\bar\varphi^*}\,,
\]
are well-defined normed spaces and satisfy $E_{\bar\varphi^*}\subset
L_{\bar\varphi^*}\subset E_{\varphi}^*$.
The following result gives a further characterization in
terms of completeness.
\begin{prop}\label{prop:dual3}
  Let $\varphi$ be a lower semicontinuous convex semi-modular on $X$
  satisfying condition~\eqref{ip_varphi}. Then, 
  it holds that
  \beq\label{ineq:norms}
  \norm{y}_{\bar\varphi^*}\leq \norm{y}_{E_{\varphi}^*}\leq 2\norm{y}_{\bar\varphi^*}
  \qquad\forall\,y\in L_{\bar\varphi^*}\,.
  \eeq
  In particular, the modular spaces $(L_{\bar\varphi^*}, \norm{\cdot}_{\bar\varphi^*})$
  and $(E_{\bar\varphi^*}, \norm{\cdot}_{\bar\varphi^*})$ are Banach spaces,
  and it holds
  \[
  E_{\bar\varphi^*}\embed L_{\bar\varphi^*}= E_{\varphi}^*\,.
  \]
  Moreover, if $s>1$ in \eqref{ip_varphi} and $E_\varphi\embed X$ densely, 
  then also $X^*\embed E_{\bar\varphi^*}$
  continuously.
\end{prop}
\begin{proof}
  {\sc Step 1}.
  Let $y\in L_{\bar\varphi^*}$:
  by definition of dual norm and 
  by the Young inequality we have
  \[
  \norm{y}_{E_{\varphi}^*}=\sup\left\{\ip{y}{x}:\;
  x\in E_\varphi\,,\;\norm{x}_{\varphi}\leq1\right\}
  \leq\bar\varphi^*(y) + \sup\left\{\varphi(x):\;
  x\in E_\varphi\,,\;\norm{x}_{\varphi}\leq1\right\}\,.
  \]
  Now, let $x\in E_\varphi$ be such that $\norm{x}_\varphi\leq 1$.
  If $(\delta_k)_k\subset(0,1)$ is such that  $\delta_k\nearrow1$ as $k\to\infty$,
  then we have that $\norm{\delta_k x}_{\varphi}<1$ for every $k$,
  hence also, by Lemma~\ref{lem:prop} (2),
  \[
  \varphi(\delta_k x)\leq \norm{\delta_k x}_\varphi=\delta_k \norm{x}_{\varphi}\leq \delta_k\,.
  \]
  Since $\delta_k x\to x$ in $X$, letting $k\to\infty$ we get, by lower semicontinuity of $\varphi$,
  \[
  \varphi(x)\leq\liminf_{k\to\infty}\varphi(\delta_k x)\leq1
  \qquad\forall\,x\in E_\varphi:\;\norm{x}_\varphi\leq1\,.
  \]
  Putting this information together, we deduce that 
  \[
  \norm{y}_{E_{\varphi}^*} \leq 1 + \bar\varphi^*(y)
  \qquad\forall\,y\in L_{\bar\varphi^*}\,.
  \]
  Now, let $\lambda>0$ be such that $\bar\varphi^*(y/\lambda)\leq1$. The inequality just proved
  implies, by arbitrariness of $y\in L_{\bar\varphi^*}$, that
  \[
  \frac1\lambda\norm{y}_{E_{\varphi}^*}=\norm{y/\lambda}_{E_{\varphi}^*}
  \leq 1+ \bar\varphi^*(y/\lambda)\leq 2\,,
  \]
  from which $\norm{y}_{E_{\varphi}^*}\leq 2\lambda$. By arbitrariness of $\lambda$
  the right-inequality in \eqref{ineq:norms} follows.
  As for the left-inequality, 
  for $y\in E_\varphi^*\setminus\{0\}$ we have
  \begin{align*}
    \bar\varphi^*\left(\frac{y}{\norm{y}_{E_\varphi^*}}\right)=
    \sup_{x\in E_\varphi}
    \left\{\frac{\ip{y}{x}_{E_\varphi^*, E_\varphi}}{\norm{y}_{E_\varphi^*}}
    -\varphi(x)\right\}\leq
    \sup_{x\in E_\varphi}
    \left\{\norm{x}_{\varphi}-\varphi(x)\right\}\,.
  \end{align*}
  Now, if $x\in E_\varphi$ and $\norm{x}_\varphi>1$, from 
  Lemma~\ref{lem:prop} (3) we know that $\varphi(x)\geq\norm{x}_\varphi$, 
  hence also $\norm{x}_{\varphi}-\varphi(x)\leq0$. It follows then that 
  \[
    \bar\varphi^*\left(\frac{y}{\norm{y}_{E_\varphi^*}}\right)\leq
    \sup_{x\in E_\varphi}
    \left\{\norm{x}_{\varphi}-\varphi(x)\right\}\leq
    \sup_{\norm{x}_\varphi\leq1}
    \left\{\norm{x}_{\varphi}-\varphi(x)\right\}\leq1\,.
  \]
  This yields that $y\in L_{\bar\varphi^*}$ and 
  $\norm{y}_{\bar\varphi^*}\leq \norm{y}_{E_\varphi^*}$, as desired.
  Hence, the inequality \eqref{ineq:norms} is proved.
  As a byproduct, this implies also that $L_{\bar\varphi^*}=E_\varphi^*$,
  and that the dual norm on $E_\varphi^*$ is equivalent to the 
  $\norm{\cdot}_{\bar\varphi^*}$-norm.\\
  Let us show now that $E_{\bar\varphi^*}$ and $L_{\bar\varphi^*}$ are Banach spaces.
  To this end, let $(y_n)_n\subset E_{\bar\varphi^*}$ be a Cauchy sequence:
  then, by \eqref{ineq:norms} it follows that it is also Cauchy in $E_\varphi^*$.
  By completeness of $E_\varphi^*$, there is $y\in E_\varphi^*$ such that 
  $y_n\to y$ in $E_{\varphi}^*$.
 Proceeding now as in the proof of Proposition~\ref{prop:banach},
 by the lower semicontinuity of $\bar\varphi^*$ on $E_\varphi^*$
 we deduce that $y\in E_{\bar\varphi^*}$, hence again by 
 \eqref{ineq:norms} that $y_n\to y$ in $E_{\bar\varphi^*}$.
 The case of $L_{\bar\varphi^*}$ is entirely analogous.
 \smallskip\\
 {\sc Step 2}.
 Let us suppose that $s>1$ and $E_\varphi\embed X$ densely,
 and show that $X^*\embed E_{\bar\varphi^*}$.
 First of all, note that we can identify $X^*$ with a closed subspace of $E_\varphi^*$:
 indeed, denoting by $i\colon E_\varphi\to X$ the inclusion, it is
 easily seen that since $E_\varphi\embed X$ densely
 the adjoint operator $i^*\colon X^*\to E_\varphi^*$ is linear, continuous, and injective.
 Hence, we can identify $X^*\cong i^*(X^*)\embed E_\varphi^*$, getting
 \[
 E_\varphi\embed L_\varphi\embed X\,, \qquad
 X^*\embed E_{\varphi}^*\,.
 \]
 Secondly, let us show that $(\bar\varphi^*)_{|X^*}\leq\varphi^*$.
 Indeed, for every $y\in X^*$ we have
  \begin{align*}
  \bar\varphi^*(y)&=\sup_{x\in E_\varphi}\{\ip{y}{x}_{E_\varphi^*, E_\varphi}-\varphi(x)\} \\
  &=\sup_{x\in E_\varphi}\{\ip{y}{x}_{X^*,X}-\varphi(x)\}\\
  &\leq
  \sup_{x\in X}\{\ip{y}{x}_{X^*,X}-\varphi(x)\}\\
  &=\varphi^*(y)\,.
  \end{align*}
  Finally, we are now ready to conclude.
  Indeed, recalling that $s>1$,
  taking conjugates in \eqref{ip_varphi} yields, after
  a standard computation,
  \[
  \varphi^*(y)\leq \frac{s-1}{s}(cs)^{-\frac1{s-1}}
  \norm{y}_{X^*}^{\frac s{s-1}} \qquad\forall\,y\in X^*\,,
  \]
  which implies that actually $\varphi^*\colon X^*\to[0,+\infty)$ 
  is everywhere defined on $X^*$, and also that
  $\varphi^*(\alpha y)<+\infty$ for every $y\in X^*$ and $\alpha>0$. 
  Since $(\bar\varphi^*)_{|X^*}\leq\varphi^*$,
  this shows that $X^*\subset E_{\bar\varphi^*}$. Moreover, 
  for every arbitrary $y\in X^*$ and for any $\lambda=\lambda(y)>0$ such that 
  \[
  \lambda\geq \left(\frac{s-1}{s}(cs)^{-\frac1{s-1}}\right)^{\frac {s-1}{s}}
  \norm{y}_{X^*}\,,
  \]
  it clearly holds that 
  \[
  \bar\varphi^*(y/\lambda)\leq
  \varphi^*(y/\lambda)\leq
  \frac{s-1}{s}(cs)^{-\frac1{s-1}}
  \norm{y}_{X^*}^{\frac s{s-1}}\lambda^{-\frac s{s-1}}\leq 1\,.
  \]
  Hence, by definition of $\norm{\cdot}_{\bar\varphi^*}$ we have then that
  \[
  \norm{y}_{\bar\varphi^*}\leq
  \left(\frac{s-1}{s}(cs)^{-\frac1{s-1}}\right)^{\frac {s-1}{s}}
  \norm{y}_{X^*} \qquad\forall\,y\in X^*\,,
  \]
  so that the inclusion $X^*\embed E_{\bar\varphi^*}$ is continuous, as required.
 \end{proof}


\section{Extended variational setting and main result}
\label{sec:general_setting}
In this section, we fix the assumptions and introduce 
the extended variational setting that will be used in the paper.
After this, we present the main well posedness result.

\subsection{Assumptions}
Throughout the paper, we will work in the following framework.

\begin{description}
  \item[H0] $H$ is a real separable Hilbert space,
  $\varphi \colon H \to [0,\infty]$ is a lower semicontinuous convex semi-modular on $H$,
  and there exist constants $c>0$ and $s>1$ such that
  \[
  \varphi(x)\geq c\norm{x}_H^s \qquad\forall\,x\in H\,.
  \]
\end{description}
  By Proposition~\ref{prop:banach} applied with 
  $\rho(z)=c z^s$, $z\geq0$, this implies that the modular 
  spaces $(E_\varphi, \norm{\cdot}_\varphi)$ and $(L_\varphi, \norm{\cdot}_\varphi)$ 
  are well-defined Banach spaces with continuous inclusions
  $E_\varphi\embed L_\varphi\embed H$.
  From now on, $H$ is identified to its dual space $H^*$ by the Riesz isomorphism,
  and norm and scalar product in $H$ are denoted by $\norm{\cdot}_H$
  and $(\cdot,\cdot)$, respectively. The convex conjugate of $\varphi$
  is defined as
  \[
  \varphi^*\colon H\to[0,+\infty]\,, \qquad
  \varphi^*(y):=\sup_{x\in H}\{(y,x)-\varphi(x)\}\,, \quad x\in H\,.
  \]
  
\begin{description}
  \item[H1] $E_\varphi$ is dense in $H$, and there exists a separable 
  reflexive Banach space $V_0\embed E_\varphi$ 
  continuously and densely, such that $\varphi$
  is bounded on bounded subsets of $V_0$.
\end{description}
  The existence of such $V_0$ is an assumption of technical nature, and 
  can be seen a separability-type requirement for $E_\varphi$.
  This is satisfied in the majority of applications (see Section~\ref{sec:applications}).
  The density of~$E_\varphi$ 
  ensures first that we can identify $H$ with a closed subspace of $E_\varphi^*$.
  More
  specifically, denoting by $i\colon E_\varphi\to H$ the (continuous) inclusion and recalling that 
  $H\cong H^*$, we have that the adjoint operator $i^*\colon H\to E_\varphi^*$
  is linear, continuous, and injective: indeed, linearity and continuity 
  follow trivially by the continuous inclusion $E_\varphi\embed H$, while 
  the injectivity is an immediate consequence of the density. Hence,
  one can identify $H$ with the closed subspace $i^*(H)\subset E_\varphi^*$.
  We have then the following continuous inclusions
  \[
  E_\varphi\embed L_\varphi\embed H\embed E_\varphi^*\,.
  \]
  We introduce the restricted semi-modular 
  \[
  \bar{\varphi}\colon E_\varphi\to[0,+\infty)\,, \qquad
  \bar\varphi:=\varphi_{|E_{\varphi}}
  \]
  and its convex conjugate
  \[
  \bar{\varphi}^{*}\colon E_{\varphi}^* \to [0,\infty]\,,\qquad
  \bar{\varphi}^{*}(y):=\sup_{x\in E_\varphi}\{\ip{y}{x}_{E_\varphi^*, E_\varphi}-\varphi(x)\}\,, 
  \quad x\in E_\varphi\,.
  \]
  By Lemma~\ref{lem:dual2} and Proposition~\ref{prop:dual3}, we know that 
  $\bar\varphi^*$ is a lower semicontinuous convex semi-modular on $E_\varphi^*$,
  with $(\bar\varphi^*)_{|H}=\varphi^*$,
  and that the modular spaces $(E_{\bar\varphi^*}, \norm{\cdot}_{\bar\varphi^*})$
  and $(L_{\bar\varphi^*}, \norm{\cdot}_{\bar\varphi^*})$ are well-defined Banach spaces, 
  with continuous inclusions 
  \[
  H\embed  E_{\bar\varphi^*}\embed L_{\bar\varphi^*}= E_\varphi^*\,.
  \]
  
  \begin{description}
  \item[H2] either one of the following conditions holds:
  \begin{description}
    \item[H2i] $E_\varphi\embed L_\varphi$ densely, or
    \item[H2ii] $H\embed L_{\bar\varphi^*}$ densely.
  \end{description}
  \end{description}
The main consequence of this assumption is that  it allows to 
properly define an extended concept of duality between the spaces 
$L_\varphi$ and $L_{\bar\varphi^*}$. Specifically, 
by {\bf H0--H1} we have the continuous inclusions
\[
V_0\embed E_\varphi\embed L_\varphi\embed H
\embed  E_{\bar\varphi^*}\embed L_{\bar\varphi^*}= E_\varphi^*\embed V_0^*\,.
\]
However, a priori it is not true that there exists a duality pairing between
$L_\varphi$ and $L_{\bar\varphi^*}$ generalizing the scalar product of $H$.
This is because $L_\varphi$ may be strictly bigger than $E_\varphi$,
and $L_{\bar\varphi^*}$ may be strictly bigger than $H$ (see Secction~\ref{sec:applications}).
Assumption {\bf H2} is fundamental as it allows to extend
the scalar product of $H$ to a duality between $L_\varphi$ and $L_{\bar\varphi^*}$.
Due to the importance of this result, we collect it in the following lemma.
\begin{lem}
  \label{lem:duality}
  Assume {\bf H0--H2}. Then, there exists a unique continuous bilinear form
  \[
  [\cdot,\cdot]\colon L_{\bar\varphi^*}\times L_\varphi\to\erre\,,
  \] 
  extending the scalar product of $H$, 
  in the sense that $[y,\cdot]\colon L_\varphi\to\erre$ and $[\cdot,x]\colon L_{\bar\varphi^*}\to\erre$
  are linear and continuous for every $y\in L_{\bar\varphi^*}$ and $x\in L_\varphi$,
  respectively, and such that 
  \begin{align*}
  [y,x]=(y,x) \qquad&\forall\,x\in L_\varphi\,,\;\forall\,y\in H\,,\\
  [y,x]=\ip{y}{x}_{E_\varphi^*,E_\varphi} \qquad&\forall\,x\in E_\varphi\,,\;\forall\,y\in L_{\bar\varphi^*}\,.
  \end{align*}
  Furthermore, the following generalized H\"older inequality holds:
  \beq\label{ineq:holder}
  [y,x]\leq2\norm{y}_{\bar\varphi^*}\norm{x}_\varphi 
  \qquad\forall\,(x,y)\in L_\varphi\times L_{\bar\varphi^*}\,.
  \eeq
\end{lem}
\begin{proof}
{\sc Case 1}. We prove here the lemma in the case of {\bf H2i}, 
i.e.~if $E_\varphi\embed L_\varphi$ densely.
Let $x\in L_{\varphi}$ and $y\in L_{\bar{\varphi}^*}= E_{\varphi}^{*}$.
If either $x=0$ or $y=0$, then we set $[y,x]:=0$. 
Let us suppose that $x\neq0$ and $y\neq0$.
Take a sequence $(x_n)_{n\in \mathbb{N}} \subset E_{\varphi}$ 
such that $x_n \to x$ in $L_{\varphi}$, and define
$$
	[y, x_n]:= \ip{y}{x_n}_{E_{\varphi}^*, E_\varphi}\,, \qquad n\in\enne\,.
$$ 
Now, for every $n$, $k\in\enne$, by the Young inequality we have
\begin{align*}
	\left|
	\ip{\frac{y}{\|y\|_{{\bar{\varphi}^*}}}}
	{\frac{x_n-x_k}{\|x_n-x_k\|_{{\varphi}}}}_{E_\varphi^*,E_{\varphi}} \right| 
	& \leq \bar{\varphi}^*\left(\frac{y}{\|y\|_{{\bar{\varphi}^*}}} \right) 
	+ \bar{\varphi}\left(\frac{x_n-x_k}{\|x_n-x_k\|_{{\varphi}}}\right) \\
	& = \bar{\varphi}^*\left(\frac{y}{\|y\|_{{\bar{\varphi}^*}}} \right) + 
	\varphi\left(\frac{x_n-x_k}{\|x_n-x_k\|_{{\varphi}}}\right) \\
	& \leq 2\,,
\end{align*}
yielding 
\begin{equation*}
	|[y, x_n-x_k]| \leq 2 \|y\|_{\bar{\varphi}^*} \|x_n-x_k\|_{\varphi} \to 0\,.
\end{equation*}
Hence, 
$([y, x_n]_{\varphi})_n$ is
a Cauchy sequence in $\erre$, and 
we can define 
$$[y, x]:= \lim_{n\to\infty}[y, x_n]\,.$$
A similar argument shows that this definition is independent of the choice 
of the approximating sequence $(x_n)_{n\in\mathbb{N}}$. Moreover, it is straightforward to check that 
$[\cdot,\cdot]$ is a continuous bilinear form extending the scalar product of $H$.
It is also the only continuous bilinear form on
$L_{\bar\varphi^*}\times L_\varphi$
doing so: indeed, if $[\cdot,\cdot]_1$
and $[\cdot,\cdot]_2$ are continuous bilinear forms extending the scalar product of~$H$,
they would coincide on $L_{\bar\varphi^*}\times E_\varphi$, hence they 
coincide on the whole $L_{\bar\varphi^*}\times L_\varphi$
 by density of $E_\varphi$ in~$L_\varphi$.
 \smallskip\\
{\sc Case 2}. In the case of {\bf H2ii}, 
i.e.~when $H\embed L_{\bar\varphi^*}$ densely,
the argument is very similar, so we omit the details.
Let $x\in L_{\varphi}\subset H$ and $y\in L_{\bar{\varphi}^*}$,
with $x\neq0$ and $y\neq0$ (otherwise the definition is trivial).
Take a sequence $(y_n)_{n\in \mathbb{N}} \subset H$ 
such that $y_n \to y$ in $L_{\bar{\varphi}^*}$, and define
$$
	[y_n, x]:= (y_n,x)\,.
$$ 
As before, 
for every $n$, $k\in\enne$, by the Young inequality we have
\begin{align*}
	\left|
	\left(\frac{y_n-y_k}{\|y_n-y_k\|_{{\bar{\varphi}^*}}},
	\frac{x}{\|x\|_{{\varphi}}}\right) \right| 
	& \leq \varphi^*\left(\frac{y_n-y_k}{\|y_n-y_k\|_{{\bar{\varphi}^*}}} \right) 
	+ \varphi\left(\frac{x}{\|x\|_{{\varphi}}}\right) \\
	& = \bar{\varphi}^*\left(\frac{y_n-y_k}{\|y_n-y_k\|_{{\bar{\varphi}^*}}} \right) + 
	\varphi\left(\frac{x}{\|x\|_{{\varphi}}}\right) \\
	& \leq 2\,,
\end{align*}
yielding 
$|[y_n-y_m, x]| \leq 2 \|y_n-y_k\|_{{\bar{\varphi}^*}} \|x\|_{{\varphi}} \to 0$.
Therefore, we can define 
$$
[y, x]:= \lim_{n\to\infty}[y_n, x]_{\varphi}\,.
$$
As before, this definition is independent of the choice 
of the approximating sequence $(y_n)_n$, and
$[\cdot,\cdot]$ is the unique continuous bilinear form 
on $L_{\bar\varphi^*}\times L_\varphi$
extending the scalar product of $H$.
\end{proof}

\begin{remark}[Assumption {\bf H0}]
  Assumption {\bf H0} is very natural in applications to 
  nonlinear evolution problems: for example, for a wide class of 
  evolutionary PDEs of parabolic type a natural choice 
  for the space $H$ is $L^2(\Omega)$, with $\Omega$ being a 
  sufficiently smooth domain in $\erre^d$, while $\varphi$ is
  the convex part of the energy driving the evolution.
  Several examples are given in Section~\ref{sec:applications}.
\end{remark}
\begin{remark}[Assumption {\bf H1}]
  As far as assumption {\bf H1} is concerned, one can easily check 
  that this is verified by a huge class of potentials $\varphi$, not necessarily of 
  polynomial growth. More importantly,
  {\bf H1} includes 
  several interesting examples in which $E_\varphi$ is not reflexive,
  hence cannot be framed in classical variational structures:
  again, a spectrum of applications is given in Section~\ref{sec:applications}.
  It excludes, however, some singular Orlicz--Sobolev spaces 
  associated, for instance, 
  to convex functions defined of bounded intervals 
  and blowing up at the extreme points. In these extreme cases,
  the space $E_\varphi$ often reduces to the trivial space $\{0\}$.  
\end{remark}
\begin{remark}[Assumption {\bf H2}]
  Let us spend a few words on the idea behind {\bf H2}.
  This hypothesis requires
  the density either of $E_\varphi$ in $L_\varphi$, or
  of $H$ in $L_{\bar\varphi^*}$. In the former case, 
  we are supposing that the modular spaces 
  $E_\varphi$ in $L_\varphi$ are ``not too different'': 
  this is verified when the potential $\varphi$ satisfies 
  a so-called $\Delta_2$-type condition (see Section~\ref{sec:applications}),
  and in such a case it actually holds that $E_\varphi=L_\varphi$.
  In the latter case, by contrast, what we are 
  supposing is that the modular space $L_{\bar\varphi^*}$
  is not ``too much bigger'' than $H$ itself. Again, this happens
  when the conjugate function $\varphi^*$ satisfies a $\Delta_2$-type condition.
  The main advantage of assumption {\bf H2} is that in the majority of
  applications either $\varphi$ or $\varphi^*$ always satisfy a $\Delta_2$-condition:
  the rough idea is that whenever $\varphi$ is not $\Delta_2$
  (for example if it grows super-polynomially at infinity) then by contrast its conjugate 
  $\varphi^*$ behaves in a $\Delta_2$-fashion (it grows sub-polynomially),
  and viceversa.
  This allows to include in assumption {\bf H2} a wide variety 
  of very singular problems, where $E_\varphi$ and $L_\varphi$
  are not necessarily reflexive.
\end{remark}

\subsection{The extended variational setting}
The main idea is to work on the triplet
\[
  L_\varphi \embed H \embed L_{\bar\varphi^*}\,,
\]
where both inclusions are continuous, the first one is also dense,
and the second one is dense if {\bf H2ii} holds.
Let us point out that such variational setting is non-standard
for the following main reasons. First, the space $L_\varphi$
is allowed to be non-reflexive, thus including several applications 
to singular PDEs of evolutionary type. Secondly, 
the space $L_{\bar\varphi^*}$ is not the dual of $L_{\varphi}$, and
the duality pairing between them is not given by the classical 
duality, but by the generalized bilinear form $[\cdot,\cdot]$.
Finally, the space $L_\varphi$ and $L_{\bar\varphi^*}$ 
are not necessarily separable.

The lack of separability and reflexivity in evolution problems is a crucial issue
that creates several difficulties. 
As for separability, the main issue concerns measurability for 
vector-valued functions: indeed, by the Pettis measurability theorem, a necessary
condition for a Banach-space-valued function to be 
Bochner-measurable is that it is essentially separably-valued
(see, for example, \cite[Sect.~II, Thm.~2]{DieUhl1977}).
This forces us to work on 
spaces of weakly-measurable functions instead.
As for reflexivity,
the main drawback that we need to face is the following.
If $X$ is a reflexive Banach space, then
it is well-known that Sobolev--Bochner spaces in the form
$W^{1,p}(a,b; X)$, where $p\geq1$ and $[a,b]\subset\erre$ is a bounded interval,
can be characterized as spaces of absolutely continuous functions
in $L^p(a,b; X)$ with almost everywhere derivative in $L^p(a,b;X)$.
In particular, the reflexivity of $X$ implies that any 
absolutely continuous function $[a,b]\to X$
is almost everywhere differentiable:
see for example \cite[Thm.~1.16]{barbu-monot}.
Nevertheless, if $X$ is not reflexive, these results are actually false,
and such characterization of $W^{1,p}(a,b; X)$ spaces is no longer valid.
In this case, one has to additionally require the almost everywhere differentiability,
as this is not granted by the absolute continuity itself (see \cite{brezis-anal} and
\cite[Thm.~8.57]{Leo2017}).
In particular, if $X$ is not reflexive, there exist
absolutely continuous functions $[a,b]\to X$
that are nowhere differentiable (e.g.~\cite[Ex.~8.30 and 8.32]{Leo2017}). 

Let us introduce some notation for the spaces of vector-valued
integrable functions that we will use. 
We refer the reader to \cite[Sect.~II]{DieUhl1977} for the general theory of integration. 
From now on, $T>0$
is a fixed final time.
We set 
\begin{align*}
  &L^1_w(0,T; L_\varphi):=
  \left\{v\colon[0,T]\to L_{\varphi}: \quad[y,v]\in L^1(0,T)\quad\forall\,y\in L_{\bar\varphi^*}
  \right\}\,,\\
  &L^1_w(0,T; L_{\bar\varphi^*}):=
  \left\{v\colon[0,T]\to L_{\bar\varphi^*}:\quad [v,x]\in L^1(0,T)\quad\forall\,x\in L_{\varphi}
  \right\}\,,\\
  &L^1(0,T; L_\varphi):=
  \left\{v\colon[0,T]\to L_{\varphi} \text{ strongly measurable}: \quad\norm{v}_\varphi\in L^1(0,T)\right\}\,,\\
  &L^1(0,T; L_{\bar\varphi^*}):=
  \left\{v\colon[0,T]\to L_{\bar\varphi^*} \text{ strongly measurable}: \quad\norm{v}_{\bar\varphi^*}
  \in L^1(0,T)\right\}\,.
\end{align*}
Let us point out that under assumption ${\bf H2i}$ we have that 
$L_\varphi$ is separable, since so is $E_\varphi$ by {\bf H1},
hence elements in $L^1_w(0,T; L_\varphi)$ are also strongly measurable in this case.
Under assumption ${\bf H2ii}$, we have instead that 
$L_{\bar\varphi^*}$ is separable, since so is $H$,
and in this case elements of $L^1_w(0,T; L_{\bar\varphi^*})$
are strongly measurable.

The following result holds also for non-reflexive spaces, 
hence are fundamental in our setting: the reader can refer to 
\cite[Thm.~1.17]{barbu-monot}
\begin{prop}\label{prop:ac}
	Let $X$ be a Banach space and $u\in L^p(a,b;X)$, with 
	$1\leq p \leq \infty$ and $[a,b]\subset\erre$.
	Then $u\in W^{1,p}(a,b;X)$ if and only if 
	there exists an absolutely continuous function 
	$u^0\colon [a,b] \to X$, i.e.~$u^0\in AC([0,T]; X)$, 
	which is almost everywhere 
	differentiable on $[a,b]$ with
	$\frac{du^0}{dt} \in L^p(0,T;X)$, 
	such that 
	$u(t) = u^0(t)$ for almost every $t\in(a,b)$.
	In such case, $\frac{du^0}{dt}$ coincides with the weak derivative $\partial_t u$
	of $u$.
\end{prop}
Inspired by Proposition~\ref{prop:ac}, it is natural to define
\begin{align*}
  W^{1,1}_w(0,T; L_{\bar\varphi^*}):=
  \bigg\{&v:[0,T]\to L_{\bar\varphi^*}: 
  \quad\exists\,v'\in L^1_w(0,T; L_{\bar\varphi^*}):\\
  &\quad[v(t),x]=[v(0),x]+\int_0^t[v'(s),x]\,ds\quad\forall\,x\in L_{\varphi}\bigg\}\,.
\end{align*}
Note that $W^{1,1}_w(0,T; L_{\bar\varphi^*})$ only
implies weak* continuity in $E_\varphi^*=L_{\bar\varphi^*}$,
and not absolute continuity as in the classical case.

\subsection{Main results}
We are ready to present here the main results of the paper.
From now on, $T>0$ is a a fixed final time.

The first result that we present is a fundamental 
computational tool in order to handle evolution equations 
in singular modular spaces. It ensures 
that under assumptions {\bf H0--\bf H2} 
the novel variational setting $(L_\varphi, H, L_{\bar\varphi^*})$
with duality given by $[\cdot, \cdot]$
is actually suited for dealing with evolution problems,
even without the classical reflexivity/separabilty assumptions. 
This is an interesting generalization to the non-reflexive
and non-separable case
of a well-know ``chain-rule'' property for vector-valued functions.

\begin{thm}[Generalized chain rule]
\label{thm:chain}
	Assume {\bf H0--H2}, and let
	\[
	u\in W^{1,1}_w(0,T;L_{\bar{\varphi}^*}) \cap 
	L^{1}_w(0,T;L_\varphi)\,,
	\]
	be such that
	\[
	\partial_tu=u_1' + u_2'\,, \qquad\text{with}\qquad
	u_1'\in L^1_w(0,T;L_{\bar{\varphi}^*})\,, \quad
	u_2'\in L^1(0,T; H)\,.
	\]
	If there exists $\alpha>0$ such that 
	\[
	\varphi(\alpha u),\;\varphi^*(\alpha u_1')\in L^1(0,T)\,,
	\]
	then $u\in C^0([0,T]; H)$,
	the function $t\mapsto\norm{u(t)}_H^2$, $t\in[0,T]$,
	is absolutely continuous, and it holds that
	\begin{equation}\label{eq:u^2}
		[\partial_t u, u] = \frac{d}{dt}\frac{1}{2} \|u\|^2_H
		\qquad\text{a.e.~in } (0,T)\,.
	\end{equation}
\end{thm}

The second result concerns existence of variational solutions
for evolution equations in modular spaces.

\begin{thm}[Existence of solutions]
  \label{thm}
  Assume {\bf H0--H2} and let 
  \beq
  \label{data}
  u_0\in H\,,\qquad
  f\in L^1(0,T; H)\,.
  \eeq
  Then, there exists a unique pair $(u,\xi)$, with 
  \begin{align}
  \label{u}
  &u \in W^{1,1}_w(0,T; L_{\bar\varphi^*})\cap C^0([0,T]; H)\cap L^1_w(0,T; L_\varphi)\,,\\
  \label{xi}
  &\xi \in L^1_w(0,T; L_{\bar\varphi^*})\,,\\
  \label{int}
  &\varphi(u),\,\bar\varphi^*(\xi) \in L^1(0,T)\,,
  \end{align}
  such that
  \beq\label{eq}
  \partial_t u + \xi = f \quad\text{in } L_{\bar\varphi^*}\quad\text{a.e.~in } (0,T)\,,
  \qquad
  u(0)=u_0\,,
  \eeq
  and
  \beq\label{sub}
  \varphi(u)+ [\xi,x-u] \leq \varphi(x) \quad\forall\,x\in E_\varphi\,, \quad\text{a.e.~in } (0,T)\,.
  \eeq
  Moreover, the following energy equality holds:
  \beq\label{energy}
  \frac12\norm{u(t)}_H^2
  +\int_0^t\left[\xi(s), u(s)\right]\,ds=
  \frac12\norm{u_0}_H^2
  +\int_0^t\left(f(s), u(s)\right)\,ds \qquad\forall\,t\in[0,T]\,.
  \eeq
  In particular, under {\bf H2i} it holds that $u\in L^1(0,T; L_\varphi)$,
  while under {\bf H2ii} it holds that $u\in W^{1,1}(0,T; L_{\bar\varphi^*})$
  and $\xi\in L^1(0,T; L_{\bar\varphi^*})$.
\end{thm}

Condition \eqref{sub}  is the natural generalization of the 
classical subdifferential inclusion
\beq
  \label{sub2}
  \xi\in\partial\varphi(u)\qquad\text{a.e.~in } (0,T)\,.
\eeq
In our setting,  $\xi$ is less regular 
than $H$, as $\xi$ only belongs to $L_{\bar\varphi^*}$, 
hence the classical differential inclusion \eqref{sub2} 
makes no sense here.
Similarly, the classical relaxation 
of the inclusion \eqref{sub2} given by 
\beq\label{sub3}
\xi\in\partial\bar\varphi(u) \qquad\text{a.e.~in } (0,T)\,,
\eeq
is useless as well, as $u$ is not necessarily $E_\varphi$-valued.
For these reasons, the introduction of the novel duality $[\cdot, \cdot]$
is crucial, as it allows to give sense to the subdifferential inclusion
in the spaces $L_\varphi$--$L_{\bar\varphi^*}$ as done in \eqref{sub}.
Clearly, this is a very natural extension of both \eqref{sub2} and \eqref{sub3}:
indeed, whenever $\xi$ is $H$-valued (or $u$ is $E_\varphi$-valued, respectively)
then \eqref{sub} is equivalent to \eqref{sub2} (or \eqref{sub3}, respectively).
An equivalent formulation of the relaxed condition \eqref{sub} is given by
\beq
  \label{sub4}
  [\xi, u] = \varphi(u) + \bar\varphi^*(\xi) \qquad\text{a.e.~in } (0,T)\,.
\eeq

Finally, the last result that we present 
is a continuous dependence result, with ensures that 
the evolution problem is actually well-posed.
\begin{thm}[Continuous dependence on the data]
  \label{thm2}
  Assume {\bf H0--H2}, and
  let $(u_0^1, f_1)$ and $(u_0^2,f_2)$ satisfy \eqref{data}.
  Then, for any respective solutions 
  $(u_1,\xi_1)$ and $(u_2,\xi_2)$
  to \eqref{u}--\eqref{sub}, it holds that
  \[
  \norm{u_1-u_2}_{C^0([0,T]; H)}^2
  +\norm{[\xi_1-\xi_2, u_1-u_2]}_{L^1(0,T)}
  \leq 2
  \left(\norm{u_0^1-u_0^2}_H^2+\norm{f_1-f_2}_{L^1(0,T; H)}^2\right)\,.
  \]
\end{thm}


\section{Proof of the generalized ``chain rule''}
\label{sec:chain}
This section is devoted to the proof of Theorem~\ref{thm:chain}.
Let us work then in the notation and setting of Theorem~\ref{thm:chain}.
The proof is organized in several steps.

\subsection{Time-regularization}
First of all, the idea is to regularize $u$ in time using convolutions.
For every $n\in\enne$, we introduce the convolution operator 
\begin{align*}
  &T_n\colon L^1(0,T)\to L^1(0,T)\,, \\ 
  &(T_nv)(t):=n \int_0^T v(s) \varrho ((t-s)n)\,ds\,, \quad t\in[0,T]\,,\quad v\in L^1(0,T)\,,
\end{align*}
where $\varrho \in C^\infty_c(\mathbb{R})$  is nonnegative
with $\int_{\erre} \varrho(s)\,ds = 1$, $\varrho(t)=\varrho(-t)$,
and $\operatorname{supp}\varrho \subset [-1,1]$.
It is well-known that $T_n$ is linear, continuous,
and sub-Markovian, in the sense that, for every $v\in L^1(0,T)$,
\[
  0\leq v\leq1 \quad\text{a.e.~in } (0,T) \qquad\Rightarrow\qquad
  0\leq T_nv\leq1 \quad\text{a.e.~in } (0,T)\,.
\]
Also, it holds that 
\[
  \|T_nv\|_{L^1(0,T)} \leq \|v\|_{L^1(0,T)} \qquad\forall\,v\in L^1(0,T)\,.
\]
Furthermore, for any Banach space $X$,
it is clear that $T_n$ can be extended to the vector-valued operator
\[
  T_n^X\colon L^1(0,T; X)\to L^1(0,T; X)\,,
\]
and it is well-known that for every $v\in L^1(0,T; X)$ it holds that 
$T_n^Xv\to v$ in $L^1(0,T; X)$ as $n\to\infty$,
see for instance \cite[Thm 8.20--8.21]{Leo2017}.

\subsection{Proof under assumption {\bf H2ii}}
Let us prove the result under assumption {\bf H2ii} first.
As we pointed out above, this implies that $L_{\bar\varphi^*}$
is separable, so that actually
\beq\label{aux0}
  u \in W^{1,1}(0,T; L_{\bar\varphi^*})\,.
\eeq
In particular, $u\colon [0,T]\to L_{\bar\varphi^*}$ is Bochner-measurable, 
hence also Borel-measurable, and $\norm{u}_{\bar\varphi^*}$
is a measurable function.
We show now that actually also
$\norm{u}_\varphi$ and $\norm{u}_H$ are measurable functions.
To this end, we use the following lemma.
\begin{lem}
\label{lem:borel}
  Set 
  \[
  \Psi_*\colon L_{\bar\varphi^*}\to[0,+\infty]\,,\qquad
  \Psi_*(y):=\begin{cases}
  \norm{y}_H \quad\text{if } y\in H\,,\\
  +\infty \quad\text{otherwise}\,,
  \end{cases}
\]
and 
\[
  \Psi_H\colon H\to[0,+\infty]\,,\qquad
  \Psi_H(y):=\begin{cases}
  \norm{y}_\varphi \quad\text{if } y\in L_\varphi\,,\\
  +\infty \quad\text{otherwise}\,.
  \end{cases}
\]
Then, $\Psi_*$ and $\Psi_H$ are convex, proper, 
and lower semicontinuous. In particular, 
$H$ is a Borel subset of $L_{\bar\varphi^*}$ and that 
$L_\varphi$ is a Borel subset of $H$.
\end{lem}
\begin{proof}
It is clear that $\Psi_*$ and $\Psi_H$ are convex and proper.
Moreover, the lower semicontinuity of
$\Psi_*$ it follows directly from the reflexivity of $H$.
As for $\Psi_H$, given $(y_n)_n\subset L_\varphi$ and $y\in H$
such that $y_n\to y$ in $H$ and $\|y_n\|_{\varphi}\leq C$
for some arbitrary constant $C>0$, we need to check that 
$y\in L_\varphi$ and $\|y\|_\varphi\leq C$.
To this end, we note that
the real sequence $(\norm{y_n}_\varphi)_n$
is bounded, so that there exists $\lambda\geq0$
and a subsequence $(y_{n_k})_k$ such that 
$\lambda_k:=\norm{y_{n_k}}_\varphi\to\lambda$ as $k\to\infty$. 
If $\lambda=0$, then trivially $y=0\in L_\varphi$.
Otherwise, if $\lambda>0$ then
$\lambda_k>\lambda/2$ for $k$ sufficiently large, so that
\[
  \limsup_{k\to\infty}\norm{\frac{y_{n_k}}{\lambda_k} - \frac y{\lambda}}_H
  \leq\frac2{\lambda}\limsup_{k\to\infty}\norm{y_{n_k}-y}_H + 
  \norm{y}_H\limsup_{k\to\infty}\left|\frac1{\lambda_k} - \frac1{\lambda}\right| =0\,.
\]
Hence $y_{n_k}/\lambda_k \to y/\lambda$ in $H$,
and by lower semicontinuity of $\varphi$
and definition of $\lambda_k$
we have
  \[
  \varphi(y/\lambda)\leq\liminf_{k\to\infty}\varphi(y_{n_k}/\lambda_k)\leq 1\,,
  \]
which implies that $y\in L_\varphi$ and $\norm{y}_\varphi\leq \lambda\leq C$.
Hence, also $\Psi_H$ is lower semicontinuous.\\
Eventually, the choice of $\Psi_*$ and $\Psi_H$ implies, by 
lower semicontinuity, that
the sets $\{y\in L_{\bar\varphi^*}: \Psi_*(y)\leq K\}$ and 
$\{x\in H: \Psi_H(x)\leq K\}$ are closed in $L_{\bar\varphi^*}$
and $H$, respectively, for all $K>0$. Hence, since
\[
  H=\bigcup_{k\in\enne}\left\{y\in L_{\bar\varphi^*}: \Psi_*(y)\leq \frac1k\right\}\,,\qquad
  L_\varphi=\bigcup_{k\in\enne}\left\{x\in H: \Psi_H(x)\leq \frac1k\right\}\,,
\]
we deduce that $H$ is a Borel subset of $L_{\bar\varphi^*}$ and that 
$L_\varphi$ is a Borel subset of $H$.
\end{proof}

Consequently, by Lemma~\ref{lem:borel},
the facts that $u$ is essentially $L_\varphi$-valued 
and $L_{\bar\varphi^*}$-Borel measurable
implies that 
$u\colon [0,T]\to L_\varphi$ 
and $u\colon [0,T]\to H$ are Borel measurable as well. In particular, 
this implies that $\norm{u}_\varphi$, $\norm{u}_H \colon [0,T]\to\erre$
are measurable. 

Thanks to {\bf H0} and
the separability of $H$, 
the condition $\varphi(\alpha u)\in L^1(0,T)$ implies that 
\[
  u \in L^s(0,T; H)\,.
\]
Moreover, noting that from Lemma~\ref{lem:prop}~(2)--(3) we have 
$\alpha\norm{u}_\varphi\leq1+\varphi(\alpha u)$, we infer that
\[
  \norm{u}_\varphi \in L^1(0,T)\,.
\]

Now, for every $n\in\enne$ we set $u_n:=T_n^Hu$ and note that 
\[
  u_n\in C^k([0,T]; H)\quad\forall\,k\in\enne\,,
\]
where, as $n\to\infty$, we have
  \begin{align}
  \label{aux1}
  u_n\to u \quad&\text{in } L^s(0,T; H)\,,\\
  \label{aux2}
  \partial_t u_n \to \partial_t u \quad&\text{in } L^1(0,T; L_{\bar\varphi^*})\,,\\
  \label{aux2'}
  T_n^Hu_2'\to u_2' \quad&\text{in } L^1(0,T; H)\,.
\end{align}
Furthermore, note that $u_n$ is essentially $L_\varphi$-valued:
indeed, thanks to the abstract Jensen inequality for 
sub-Markovian operators (see Haase \cite[Thm.~3.4]{haase}),
for almost every $t\in(0,T)$
we have that 
\[
  \varphi(\alpha u_n(t))=\varphi(T_n^H(\alpha u)(t)) \leq
  T_n\left[\varphi(\alpha u)\right](t)\,.
\]
Since $\varphi(\alpha u)\in L^1(0,T)$, by definition of $T_n$
one has also that $T_n(\varphi(\alpha u))\in L^1(0,T)$ for all $n\in\enne$, 
which yields in turn by comparison that $\varphi (\alpha u_n)\in L^1(0,T)$
for all $n\in\enne$. This clearly implies that $u_n$ is essentially $L_\varphi$-valued
and, after integration in time and by contraction of $T_n$ in $L^1(0,T)$, that
\[
  u_n\colon (0,T)\to L_\varphi\,, \qquad 
  \norm{\varphi(\alpha u_n)}_{L^1(0,T)}\leq\norm{\varphi(\alpha u)}_{L^1(0,T)}\,.
\]
Furthermore, 
for every $y\in H$ and almost every $t\in(0,T)$,
by definition of $T_n$ and $T_n^H$ we have that 
\[
  (y, u_n(t))=(y, T_n^Hu(t))=T_n(y, u(t)) \qquad\forall\,y\in H\,.
\]
As $H$ is dense in $L_{\bar\varphi^*}$ by {\bf H2ii}, this implies that
\[
  [y, u_n(t)] = T_n[y,u(t)] \qquad\forall\,y\in L_{\bar\varphi^*}\,,\quad\text{for a.e.~$t\in(0,T)$}\,.
\]
Since $[y,u]\in L^1(0,T)$ for every $y\in L_{\bar\varphi^*}$, letting $n\to\infty$ we deduce that 
\beq
\label{aux3}
  [y, u_n] \to [y,u] \quad\text{in } L^1(0,T)\,, \quad\forall\,y\in L_{\bar\varphi^*}\,.
\eeq

Now, since in particular  $u_n\in C^1([0,T]; H)$, 
it is well-known (see \cite[Thm.~1.9]{barbu-monot}) that 
\[
\frac{d}{dt}\frac{1}{2} \|u_n\|^2_H=
(\partial_t u_n, u_n)=
[\partial_t u_n, u_n] \qquad\text{in } [0,T]\,,
\]
yielding, after integration in time
\beq\label{eq:vt=vs+int}
  \frac12\norm{u_n(t)}_H^2 - \frac12\norm{u_n(s)}_H^2 = 
  \int_s^t[\partial_t u_n(r), u_n(r)]\, dr\,, \qquad
  \forall\,s,t\in[0,T]\,.
\eeq
By  the abstract Jensen inequality,
since $T_n^{L_{\bar\varphi^*}}$ coincides with $T_n^H$ on $L^1(0,T; H)$,
\begin{align*}
  &\int_s^t[\partial_t u_n(r), u_n(r)]\, dr
  \leq
  \int_0^T|[\partial_t u_n(r), u_n(r)]|\, dr\\
  &\leq\frac1{\alpha^2}\int_0^T\varphi(\alpha u_n) + 
  \frac1{\alpha^2}\int_0^T\bar\varphi^*(\alpha T_n^{L_{\bar\varphi^*}} u_1')
  +\int_0^T\norm{T_n^Hu_2'(s)}_H\norm{u_n(s)}_H\,ds\\
  &\leq\frac1{\alpha^2}\int_0^T\varphi(\alpha u) + 
  \frac1{\alpha^2}\int_0^T\bar\varphi^*(\alpha u_1')
  +\int_0^T\norm{u_2'(s)}_H\norm{u_n(s)}_H\,ds\,.
\end{align*}
Moreover, using the H\"older inequality and the 
weighted Young inequality 
on the last term of the right-hand side in the form 
\[
  ab\leq\frac14 a^2 + b^2\qquad\forall\,a,b\geq0\,,
\]
we deduce that 
\begin{align*}
  \int_s^t[\partial_t u_n(r), u_n(r)]\, dr
  &\leq \frac1{\alpha^2}\int_0^T\varphi(\alpha u) + 
  \frac1{\alpha^2}\int_0^T\bar\varphi^*(\alpha u_1')
  +\norm{u_n}_{L^\infty(0,T; H)}\norm{u_2'}_{L^1(0,T; H)}\\
  &\leq \frac1{\alpha^2}\int_0^T\varphi(\alpha u) + 
  \frac1{\alpha^2}\int_0^T\bar\varphi^*(\alpha u_1')
  +\frac{1}4\norm{u_n}^2_{L^\infty(0,T; H)} + \norm{u_2'}^2_{L^1(0,T; H)}\,.
\end{align*}
Taking supremum in $t\in[0,T]$ in \eqref{eq:vt=vs+int} and
rearranging the terms, we obtain then
\[
  \frac14\norm{u_n}_{L^\infty(0,T; H)}^2\leq
  \frac12\norm{u_n(s)}_H^2 + C \qquad\forall\,s\in[0,T]
\]
where
\[
  C:=\frac1{\alpha^2}\int_0^T\varphi(\alpha u) + 
  \frac1{\alpha^2}\int_0^T\bar\varphi^*(\alpha u_1')
  + \norm{u_2'}^2_{L^1(0,T; H)}
\]
is independent of $n\in\enne$.
Taking square roots at both sides, and integrating with respect 
to $s$ on $(0,T)$ we get
\[
  \frac T2\norm{u_n}_{L^\infty(0,T; H)} \leq 
  \frac1{\sqrt2}\norm{u_n}_{L^1(0,T; H)} + \sqrt C T \leq
  \frac1{\sqrt2}\norm{u}_{L^1(0,T; H)} + \sqrt C T\,,
\]
from which 
we infer that 
\[
  \norm{u_n}_{L^\infty(0,T; H)}\leq C\,,
\]
so that $u\in L^\infty(0,T; H)$ and 
\beq
  \label{aux4}
  u_n\wstarto u \quad\text{in } L^\infty(0,T; H)\,.
\eeq

Now, we have 
\[
  [\partial_t u_n, u_n]=[T_n^{L_{\bar\varphi^*}}u_1', u_n] +(T_n^Hu_2', u_n)\,,
\]
where thanks to \eqref{aux4} and \eqref{aux2'} it holds
\[
  \int_s^t(T_n^Hu_2'(r), u_n(r))\,dr \to \int_s^t(u_2'(r), u(r))\,dr \qquad\forall\,s,t\in[0,T]\,.
\]
Moreover, 
thanks to the convergences \eqref{aux2} and \eqref{aux3} of $(u_n)_n$,
we have that, possibly extracting a non-relabelled subsequence,
\[
  [y,u_n] \to [y,u] \quad\forall\,y\in L_{\bar\varphi^*}\quad\text{a.e.~in } (0,T)\,,
  \qquad\quad
  T_n^{L_{\bar\varphi^*}}u_1' \to u_1' \quad\text{in } L_{\bar\varphi^*}
  \quad\text{a.e.~in } (0,T)\,,
\]
and similarly, since $\norm{u}_\varphi\in L^1(0,T)$, that
\[
  T_n\norm{u}_\varphi\to\norm{u}_\varphi \quad\text{a.e.~in } (0,T)\,.
\]
By the the H\"older inequality \eqref{ineq:holder}, 
using again the abstract Jensen inequality applied to the convex function 
$\norm{\cdot}_\varphi$, we also have then
almost everywhere on $(0,T)$
\begin{align*}
	\left|[T_n^{L_{\bar\varphi^*}}u_1', u_n] - [u_1', u]\right|
	&\leq
	\left|[T_n^{L_{\bar\varphi^*}}u_1'-u_1', u_n]\right| +\left|[u_1', u-u_n]\right| \\
	&\leq
	2\|T_n^{L_{\bar\varphi^*}}u_1'-u_1'\|_{{\bar{\varphi}^*}} \|T_n^Hu\|_{\varphi} 
	+\left|[u_1', u-u_n]\right|\\
	&\leq
	2\|T_n^{L_{\bar\varphi^*}}u_1'-u_1'\|_{{\bar{\varphi}^*}} T_n\|u\|_{\varphi}
	+\left|[u_1', u-u_n]\right|\to 0\,,
\end{align*}
yielding 
\[
  [T_n^{L_{\bar\varphi^*}}u_1', u_n] \to [u_1', u] \qquad\text{a.e.~in } (0,T)\,.
\]
Moreover, thanks to the 
Young inequality and the abstract Jensen inequality for 
sub-markovian operators (see again \cite[Thm.~3.4]{haase}),
\begin{align*}
  \pm\alpha^2[T_n^{L_{\bar\varphi^*}}u_1', u_n]&\leq
  \varphi(\pm \alpha u_n) + \bar\varphi^*(\alpha T_n^{L_{\bar\varphi^*}}u_1') 
  =\varphi(\alpha T_n^Hu) + \bar\varphi^*(\alpha T_n^{L_{\bar\varphi^*}}u_1')\\
  &\leq T_n\varphi(\alpha u) + T_n\bar\varphi^*(\alpha u_1')\,.
\end{align*}
This implies that 
\[
  |[T_n^{L_{\bar\varphi^*}}u_1', u_n]|\leq\frac1{\alpha^2} T_n\left(\varphi(\alpha u)
  +\bar\varphi^*(\alpha u_1')\right) \quad\text{a.e.~in } (0,T)\,,\qquad\forall\,n\in\enne\,.
\]
Since by assumption $\varphi(\alpha u)+\bar\varphi^*(\alpha u_1') \in L^1(0,T)$,
the right-hand side of such inequality converges in $L^1(0,T)$, 
hence in particular is uniformly integrable on $(0,T)$.
By comparison, we infer that the sequence $([T_n^{L_{\bar\varphi^*}}u_1', u_n])_n$
in uniformly integrable on $(0,T)$ as well. By Vitali's dominated convergence theorem 
we obtain 
\[
  [T_n^{L_{\bar\varphi^*}}u_1', u_n] \to [u_1', u] \qquad\text{in } L^1(0,T)\,.
\]
taking these remarks into account and
letting now $n\to\infty$ in \eqref{eq:vt=vs+int}, we obtain that
\[
  \frac12\norm{u(t)}_H^2 - \frac12\norm{u(s)}_H^2 = 
  \int_s^t[\partial_t u(r), u(r)]\, dr \qquad\text{for a.e.~}s,t\in(0,T)\,.
\]
Clearly, this implies that $t\mapsto\norm{u(t)}_H^2$
is absolutely continuous on $[0,T]$.
In particular, the equality holds for every $s$, $t\in[0,T]$.
Moreover, since $u\in C^0([0,T]; E_\varphi^*)$
by \eqref{aux0}, 
we have that 
the function $t\mapsto(u(t), x)$ is continuous for every $x\in E_\varphi$.
Now, for any $\bar t\in[0,T]$ and $(t_k)_k\subset[0,T]$ such that 
$t_k\to \bar t$ as $k\to\infty$, since $u\in L^\infty(0,T; H)$
on a non-relabelled subsequence we have 
$u(t_k)\wto y$ in $H$ for a certain $y\in H$. Moreover,
for any $x\in E_\varphi$ it holds that $(u(t_k), x)\to (u(\bar t), x)$,
from which $(y,x)=(u(\bar t), x)$. As $E_\varphi$ is dense in $H$,
we deduce that $y=u(\bar t)$. This shows that 
$u\colon[0,T]\to H$ is weakly continuous.
As we have already proved that $t\mapsto\norm{u(t)}_H$ is continuous,
we infer that $u\in C^0([0,T]; H)$.

\subsection{Proof under assumption {\bf H2i}} 
Let us consider now the case of assumption {\bf H2i}. 
The proof is very similar to the case {\bf H2ii}, the main difference being 
that the roles of $L_{\varphi}$ and $L_{\bar\varphi^*}$ are exchanged.
Indeed, under {\bf H2i} we have that $L_\varphi$ is separable, hence
$u\in L^1(0,T; L_\varphi)$ and we can set for any $n\in\enne$ $u_n:=T_n^{L_\varphi}u$, getting
in particular 
\[
  u_n\to u \quad\text{in } L^1(0,T; L_\varphi)\,.
\]
Proceeding as before one can also show that $u\in L^\infty(0,T; H)$ and 
\[
u_n\wstarto u\quad\text{in } L^\infty(0,T; H)\,.
\]
Moreover, $\partial_tu_n=T_n^{L_{\bar\varphi^*}}u_1'+T_n^Hu_2'$, where 
\[
  T_n^Hu_2'\to u_2' \quad\text{in } L^1(0,T; H)
\]
and
\begin{align*}
  &[T_n^{L_{\bar\varphi^*}}u_1', x]\to[u_1', x] \quad\text{in } L^1(0,T)\,,
  \quad\forall\,x\in L_\varphi\,,\\
  &\norm{T_n^{L_{\bar\varphi^*}}u_1'}_{\bar\varphi^*}
  \leq T_n\norm{u_1'}_{\bar\varphi^*} \quad\text{a.e.~in } (0,T)\,.
\end{align*}
On a not relabelled subsequence, 
by the H\"older inequality we have the almost everywhere convergence
\begin{align*}
  \left|[T_n^{L_{\bar\varphi^*}}u_1', u_n] - [u_1', u]\right|
	&\leq
	\left|[T_n^{L_{\bar\varphi^*}}u_1', u_n-u]\right| +
	\left|[T_n^{L_{\bar\varphi^*}}u_1'-u_1', u]\right| \\
	&\leq
	2\|T_n^{L_{\bar\varphi^*}}u_1'\|_{{\bar{\varphi}^*}} \|u_n-u\|_{\varphi} 
	+ |[T_n^{L_{\bar\varphi^*}}u_1'-u_1',u]| \to 0\,,
\end{align*}
Hence, writing $[\partial_t u_n, u_n]=[T_n^{L_{\bar\varphi^*}}u_1', u_n] +(T_n^Hu_2', u_n)$,
on the one hand we have again
\[
  \int_s^t(T_n^Hu_2'(r), u_n(r))\,dr \to \int_s^t(u_2'(r), u(r))\,dr \qquad\forall\,s,t\in[0,T]\,,
\]
and on the other hand,
proceeding as before using the abstract Jensen inequality and the Vitali
convergence theorem, we infer that 
\[
  [T_n^{L_{\bar\varphi^*}}u_1', u_n] \to [u_1', u] \qquad\text{in } L^1(0,T)\,.
\]
This allows to pass to the limit as $n\to\infty$ as
in the {\sc Case {\bf H2ii}} and obtain
\[
  \frac12\norm{u(t)}_H^2 - \frac12\norm{u(s)}_H^2 = 
  \int_s^t[\partial_t u(r), u(r)]\, dr \qquad\text{for a.e.~}s,t\in(0,T)\,.
\]
Hence, $t\mapsto\norm{u(t)}_H^2$
is absolutely continuous on $[0,T]$, and $u\in L^\infty(0,T; H)$.
In particular, the equality holds for every $s$, $t\in[0,T]$.
Moreover, since now $u\in W^{1,1}_w(0,T; L_{\bar\varphi^*})$,
we only have that $u$ is weakly continuous in $E_\varphi^*$.
Still, this ensures that
the function $t\mapsto(u(t), x)$ is continuous for every $x\in E_\varphi$.
As $u\in L^\infty(0,T; H)$ and $E_\varphi$ is dense in $H$, 
this implies that
$u\colon[0,T]\to H$ is weakly continuous.
As we have already proved that $t\mapsto\norm{u(t)}_H$ is continuous,
we infer that $u\in C^0([0,T]; H)$. This concludes the proof
of Theorem~\ref{thm:chain}.


\section{Proof of well-posedness}
\label{sec:proof}
This section is devoted to the proof of well-posedness
contained in Theorems~\ref{thm}--\ref{thm2}.

\subsection{The approximation}
Let us denote by $A:=\partial\varphi\colon H\to2^{H}$ the subdifferential of $\varphi$.
We recall the Young inequality
\[
  (y,x)\leq\varphi(x)+\varphi^*(y)\qquad\forall\,x,y\in H,
\]
and point out that the equality holds if and only if $y\in A(x)$.

For every $\lambda>0$, let $\varphi_\lambda\colon H\to[0,+\infty)$
be the Moreau--Yosida regularization of $\varphi$, defined as
\[
  \varphi_\lambda(x):=\inf_{y\in H}\left\{\varphi(y) + \frac1{2\lambda}\norm{x-y}_H^2\right\}\,,
  \qquad x\in H.
\]
From classical results of convex and monotone analysis (see \cite[Ch.~2]{barbu-monot}),
we have that $\varphi_\lambda\in C^1(H)$, with $D\varphi_\lambda=A_\lambda$,
where $A_\lambda\colon H\to H$ is the Yosida approximation of $A$.
Let us recall that $A_\lambda$ is defined as
\[
  A_\lambda(x):=\frac{x-J_\lambda(x)}{\lambda}, 
  \qquad x\in H,
\]
where we have denoted by $J_\lambda\colon H\to H$ the resolvent of $A$, namely
\[
  J_\lambda(x):=(I+\lambda A)^{-1}(x)\,, \qquad x\in H.
\]
It is well known that $A_\lambda$ is $\frac1\lambda$-Lipschitz continuous,
$J_\lambda$ is $1$-Lipschitz continuous,
and that $A_\lambda(x)\in A(J_\lambda(x))$
for every $x\in H$.
 Moreover, 
$\varphi_\lambda$ satisfies 
\begin{align*}
  &\varphi(J_\lambda(x))\leq \varphi_\lambda(x)\leq\varphi(x) \qquad\forall\,x\in H,\\
  &\lim_{\lambda\searrow0}\varphi_\lambda(x)=\varphi(x)\qquad\forall\,x\in H.
\end{align*}

We study the approximated problem 
\[
  \begin{cases}
  \partial_t u_\lambda + A_\lambda(u_\lambda) = f,\\
  u_\lambda(0)=u_0.
  \end{cases}
\]
Since $A_\lambda$ is Lipschitz-continuous, from the classical theory 
of nonlinear evolution equations (see again \cite{barbu-monot} or 
\cite[Prop.~3.4]{brezis})
such approximated problem admits a unique solution 
\[
  u_\lambda \in W^{1,1}(0,T, H)\,.
\]

\subsection{Uniform estimates}
Let us prove some uniform estimates on $(u_\lambda)_\lambda$,
independent of $\lambda$.
To this end, testing the approximated equation by $u_\lambda$
and integrating in time yields, for every $t\in[0,T]$,
\[
  \frac12\norm{u_\lambda(t)}_H^2
  +\int_0^t\left(A_\lambda(u_\lambda(s)), u_\lambda(s)\right)\,ds=
  \frac12\norm{u_0}_H^2
  +\int_0^t\left(f(s), u_\lambda(s)\right)\,ds.
\]
Now, on the left-hand side, since $A_\lambda(u_\lambda)\in A(J_\lambda(u_\lambda))$,
by the Young inequality we have 
\begin{align*}
  (A_\lambda(u_\lambda), u_\lambda)&=
  (A_\lambda(u_\lambda), J_\lambda(u_\lambda)) + \lambda\norm{A_\lambda(u_\lambda)}_H^2\\
  &=\varphi(J_\lambda(u_\lambda)) + \varphi^*(A_\lambda(u_\lambda))
  + \lambda\norm{A_\lambda(u_\lambda)}_H^2\,.
\end{align*}
Consequently, we have
\begin{align*}
  \frac12\norm{u_\lambda(t)}_H^2
  &+\int_0^t\varphi(J_\lambda(u_\lambda(s)))\,ds
  +\int_0^t\varphi^*(A_\lambda(u_\lambda(s)))\,ds
  +\lambda\int_0^t\norm{A_\lambda(u_\lambda(s))}_H^2\,ds\\
  &=\frac12\norm{u_0}_H^2
  +\int_0^t\left(f(s), u_\lambda(s)\right)\,ds 
  \\
  &\leq\frac12\norm{u_0}_H^2+\int_0^t\norm{f(s)}_H\norm{u_\lambda(s)}_H\,ds\,.
\end{align*}
The Gronwall lemma implies then that
there exists 
$M>0$, independent of $\lambda$, such that
\begin{align}
  \label{est1}
  \norm{u_\lambda}_{C^0([0,T]; H)}^2&\leq M,\\
  \label{est2}
  \norm{\varphi(J_\lambda(u_\lambda))}_{L^1(0,T)}+
  \norm{\varphi^*(A_\lambda(u_\lambda))}_{L^1(0,T)}+
  \lambda\norm{A_\lambda(u_\lambda)}^2_{L^2(0,T; H)}&\leq M.
\end{align}
Recalling also assumption {\bf H0}, this implies that 
\beq
  \label{est3}
  \norm{J_\lambda(u_\lambda)}_{L^s(0,T; H)}^s\leq M\,.
\eeq

Now, by the estimates \eqref{est1} and 
\eqref{est3} there exist $u\in L^\infty(0,T; H)$ and $\tilde u\in L^s(0,T; H)$ such that,
as $\lambda\searrow0$,
\begin{align}
  \label{conv1}
  u_\lambda\wstarto u \quad&\text{in } L^\infty(0,T; H),\\
  \label{conv2}
  J_\lambda(u_\lambda)\wto \tilde u\quad&\text{in } L^s(0,T; H).
\end{align}
Moreover, note that by \eqref{est2} and the definition of $A_\lambda$ we have 
\[
  \norm{J_\lambda(u_\lambda)-u_\lambda}_{L^2(0,T; H)}=
  \lambda\norm{A_\lambda(u_\lambda)}_{L^2(0,T; H)}\leq M\lambda^{1/2}\to 0\,,
\]
which implies that $\tilde u=u$ and
\beq
  \label{conv3}
  J_\lambda(u_\lambda) \wto u \quad\text{in } L^s(0,T; H)\,.
\eeq
By the weak lower semicontinuity of convex integrands,
convergence \eqref{conv3}, and estimate \eqref{est2},
we deduce then that
\[
  \int_0^T\varphi(u(s))\,ds\leq\liminf_{\lambda\searrow0}
  \int_0^T\varphi(J_\lambda(u_\lambda(s)))\,ds\leq M\,.
\]
It follows that $\varphi(u)\in L^1(0,T)$, so that
$u$ is essentially $L_\varphi$-valued.
Since by Lemma~\ref{lem:borel}
$L_\varphi$ is a Borel subset of $H$ and 
$u$ is strongly measurable in $H$, we infer that $u$
is Borel-measurable in $L_\varphi$: hence, 
by definition of Borel-measurability we have that 
$\|u\|_\varphi$ is measurable. Furthermore, 
as a consequence of Lemma~\ref{lem:prop} we have that 
\[
\|u\|_\varphi \leq 1+\varphi (u) \quad\text{a.e.~in } (0,T)\,,
\]
so that by comparison $\norm{u}_\varphi\in L^1(0,T)$.
In order to prove that $u\in L^1_w(0,T; L_\varphi)$, we need to 
show that $[y,u]\in L^1(0,T)$ for every $y\in L_{\bar\varphi^*}$.
To this end, since $u:(0,T)\to L_\varphi$ is
Borel-measurable, by definition of weak topology on $L_\varphi$
it follows that $u:(0,T)\to L_\varphi$ is weakly measurable
in the classical sense, i.e.~$\langle y,u\rangle_{L_{\varphi}^*, L_{\varphi}}$
is measurable for every $y\in L_{\varphi}^*$.
Now, under assumption {\bf H2i} 
we have that $L_\varphi$ is separable (because 
so is $E_\varphi$), hence by the Pettis theorem \cite[Thm.~2, Ch.~II]{DieUhl1977}
we obtain that $u:(0,T)\to L_\varphi$ is strongly measurable:
it follows that in this case we have actually that $u\in L^1(0,T; L_\varphi)$.
In particular, since $[y,\cdot]:L_\varphi\to\erre$ is linear continuous 
by Lemma~\ref{lem:duality}, 
by composition we infer that also $[y,u]:(0,T)\to\erre$ is measurable,
hence $u\in L^1_w(0,T; L_\varphi)$.
Alternatively, under assumption {\bf H2ii} we observe that 
$[y,u]\in L^1(0,T)$ for every $y\in H$ since $u\in L^1(0,T;H)$:
hence, the density of $H$ in $L_{\bar\varphi^*}$
readily implies that $[y,u]\in L^1(0,T)$ also for every $y\in L_{\bar\varphi^*}$, and
this shows indeed that $u\in L^1_w(0,T; L_\varphi)$.

In order to deduce some compactness for $(A_\lambda(u_\lambda))_\lambda$,
we need the following lemma.
\begin{lem}
  \label{lem:V0}
  For any reflexive Banach space $V_0$ such that $V_0\embed E_\varphi$
  continuously and densely, it holds that $E_\varphi^*\embed V_0^*$ continuously
  and densely.
  Furthermore, the convex conjugate of $\varphi_0:=\varphi_{|V_0}\colon V_0\to[0,+\infty)$ is
  given by
  \[
  \varphi_0^*\colon V_0^*\to[0,+\infty]\,,\qquad
  \varphi_0^*(y)=\begin{cases}
  \bar\varphi^*(y) \quad&\text{if } y\in E_\varphi^*\,,\\
  +\infty\quad&\text{if } y\in V_0^*\setminus E_\varphi^*\,.
  \end{cases}
  \]
  If also $\varphi$ is bounded on bounded sets of $V_0$, then
  \beq\label{super}
  \lim_{\norm{y}_{V_0^*}\to+\infty}\frac{\varphi_0^*(y)}{\norm{y}_{V_0^*}}=+\infty\,.
  \eeq
\end{lem}
\begin{proof}
  The fact that $E_\varphi^*\embed V_0^*$ continuously is an immediate consequence 
  of the density of $V_0$ in $E_\varphi$, while the fact that 
  $E_\varphi^*\embed V_0^*$ densely follows from the 
  reflexivity of $V_0$ by a classical argument.\\
  Let us compute
  the convex conjugate of $\varphi_0$.
  If $V_0=E_\varphi$, then the conclusion is trivial,
  so let us suppose then that $V_0\subset E_\varphi$ strictly.
  First of all, we show $\varphi_0^*=+\infty$ on $V_0^*\setminus E_\varphi^*$.
  Let $y\in V_0^*\setminus E_\varphi^*$: this means that $y\colon V_0\to\erre$
  cannot be extended to a continuous linear functional on $E_\varphi$,
  i.e.~there is no constant $C>0$ such that $\ip{y}{x}_{V_0^*,V_0}\leq C\norm{x}_{\varphi}$
  for all $x\in V_0$. Moreover, for all $x\in E_\varphi\setminus V_0$, 
  by density of $V_0$ in $E_\varphi$ there 
  is a sequence $(x_n)_n\subset V_0$
  such that $x_n\to x$ in $E_\varphi$. 
  If for all $x\in E_\varphi\setminus V_0$ there exists $C_x>0$ such that 
  $\ip{y}{x_n}_{V_0^*,V_0}\leq C_x\norm{x_n}_{\varphi}$ for all $n\in\enne$,
  then, due to the continous embedding $V_0 \embed E_\varphi$, 
  one could extend by density $y$ to a continuous linear functional on $E_\varphi$.
  However, this is not possible since $y\in V_0^*\setminus E_\varphi^*$:
  consequently, there exists $\bar x\in E_\varphi\setminus V_0$
  and a sequence $(x_n)_n\subset V_0$
  such that
  \[
  x_n\to \bar x\quad\text{in } E_\varphi\,, \qquad
  \ip{y}{x_n}_{V_0^*,V_0}> n\norm{x_n}_{\varphi} \quad\forall\,n\in\enne\,.
  \]
  In particular, $\varphi(x_n)\to\varphi(\bar x)$ and, since $\bar x\neq0$, 
 $\ip{y}{x_n}_{V_0^*,V_0}\to+\infty$: hence
  \begin{align*}
  \varphi_0^*(y)&=\sup_{x\in V_0}\left\{\ip{y}{x}_{V_0^*, V_0}-\varphi_0(x)\right\}
  \geq\sup_{n\in\enne}\left\{\ip{y}{x_n}_{V_0^*, V_0}-\varphi(x_n)\right\}\\
  &\geq\limsup_{n\to\infty}\left(\ip{y}{x_n}_{V_0^*, V_0}-\varphi(x_n)\right)\\
  &=\limsup_{n\to\infty}\ip{y}{x_n}_{V_0^*, V_0} - \varphi(\bar x) = +\infty\,.
  \end{align*}
  This shows that $\varphi_0^*=+\infty$ on $V_0^*\setminus E_\varphi^*$.
  Let us prove now that $(\varphi_0^*)_{|E_\varphi^*}=\bar\varphi^*$.
  Let $y\in E_\varphi^*$ arbitrary:
  we have
  \begin{align*}
  \varphi_0^*(y)=\sup_{x\in V_0}\left\{\ip{y}x_{V_0^*,V_0}-\varphi(x)\right\}
  \leq\sup_{x\in E_\varphi}\left\{\ip{y}x_{E_\varphi^*, E_\varphi}-\varphi(x)\right\}=
  \bar\varphi^*(y)\,.
  \end{align*}
  On the other hand, for all $x\in E_\varphi$ there is $(x_n)_n\subset V_0$
  such that $x_n\to x$ in $E_\varphi$, so that 
  \[
  \ip{y}{x_n}_{E_\varphi^*, E_\varphi}=
  \ip{y}{x_n}_{V_0^*,V_0}\leq
  \varphi_0^*(y) + \varphi(x_n) \qquad\forall\,n\in\enne\,, 
  \]
  hence also, noting that $\varphi(x_n)\to\varphi(x)$,
  \[
    \ip{y}{x}_{E_\varphi^*, E_\varphi}\leq\varphi_0^*(y) + \varphi(x) \qquad\forall\,x\in E_\varphi\,.
    \]
  It follows that 
  \[
  \varphi_0^*(y) \geq \sup_{x\in E_\varphi}\left\{
  \ip{y}{x}_{E_\varphi^*, E_\varphi} -\varphi(x) \right\} = \bar\varphi^*(y)\,.
  \]
  The shows that $(\varphi_0^*)_{|E_\varphi^*}=\bar\varphi^*$, as required.\\
  Finally, let us show that $\varphi_0^*$ is superlinear at $\infty$. To this end,
  by the Young inequality we have that 
  \[
  \varphi_0^*(y)\geq
  \ip{y}{x}_{V_0^*, V_0} - \varphi(x) \qquad\forall\,x\in V_0\,,\quad\forall\,y\in V_0^*\,.
  \]
  Since $V_0$ is reflexive, for any $y\in V_0^*\setminus\{0\}$, there is $x_y\in V_0$
  such that $\ip{y}{x_y}_{V_0^*, V_0}=\norm{y}_{V_0^*}^2=\norm{x_y}_{V_0}^2$. 
  Choosing
  $x=Lx_y\norm{y}_{V_0^*}^{-1}$ for arbitrary $L>0$ in the last inequality yields 
  \[
  \varphi_0^*(y)\geq L\norm{y}_{V_0^*} - \varphi\left(Lx_y\norm{y}_{V_0^*}^{-1}\right)
  \qquad\forall\,y\in V_0^*\setminus\{0\}\,,\quad\forall\,L>0\,,
  \]
  where
  \[
  \norm{Lx_y\norm{y}_{V_0^*}^{-1}}_{V_0} = L \qquad\forall\,y\in V_0^*\setminus\{0\}\,,
  \quad\forall\,L>0\,.
  \]
  Since $\varphi$ is bounded on bounded subsets of $V_0$,
  there exists $C_L>0$ such that 
  \[
  \varphi_0^*(y)\geq L\norm{y}_{V_0^*} - C_L
  \qquad\forall\,y\in V_0^*\setminus\{0\}\,,\quad\forall\,L>0\,.
  \]
  Hence, for all arbitrary $K>0$,
  choosing $L=2K$, for all $y\in V_0^*$ with 
  $\norm{y}_{V_0^*}\geq C_{2K}/K$ we have
  \[
  \frac{\varphi_0^*(y)}{\norm{y}_{V_0^*}}\geq 2K - \frac{C_{2K}}{\norm{y}_{V_0^*}}\geq K\,.
  \]
  As $K>0$ is arbitrary, we can conclude.
\end{proof}

By assumption {\bf H1}, Lemma~\ref{lem:V0}, and estimate \eqref{est2},
we deduce that 
\[
  \int_0^T\varphi_0^*(A_\lambda(u_\lambda(s)))\,ds\leq M\,,
\]
where
\[
  \lim_{\norm{y}_{V_0^*}\to+\infty}\frac{\varphi_0^*(y)}{\norm{y}_{V_0^*}}=+\infty\,.
\]
In particular, there exists an
increasing sequence $(r_n)_n$ of positive numbers 
such that 
\[
  \varphi_0^*(y)\geq n \norm{y}_{V_0^*} \qquad\forall\,y\in V_0^*\,,\quad\norm{y}_{V_0^*}\geq r_n\,,
  \qquad\forall\,n\in\enne\,.
\]
This readily implies that $(A_\lambda(u_\lambda))_\lambda$ is 
bounded in $L^1(0,T; V_0^*)$. 
Moreover, 
for any measurable $I\subset[0,T]$ we have
\begin{align*}
  \int_{I}\norm{A_\lambda(u_\lambda))}_{V_0^*} &=
  \int_{I\cap\{\norm{A_\lambda(u_\lambda))}_{V_0^*}< r_n\}}
  \norm{A_\lambda(u_\lambda))}_{V_0^*} +
  \int_{I\cap\{\norm{A_\lambda(u_\lambda))}_{V_0^*}\geq r_n\}}
  \norm{A_\lambda(u_\lambda))}_{V_0^*}\\
  &\leq |I| r_n + \frac{1}n\int_I\varphi_0^*(A_\lambda(u_\lambda))
  \leq |I| r_n + \frac{M}n\,. 
\end{align*}
Hence, for any arbitrary $\eps>0$, choosing $\bar n=\bar n(\eps)$ sufficiently large 
such that $M/\bar n\leq\eps/2$, and setting $\delta=\delta(\eps):=\eps r_{\bar n}^{-1}/2$,
we have that 
\[
  \sup_{\lambda>0}\int_I\norm{A_\lambda(u_\lambda))}_{V_0^*} \leq \eps
  \qquad\forall\,I\subset[0,T]\,,\quad|I|\leq\delta\,.
\]
This implies that the family $(A_\lambda(u_\lambda))_\lambda$
is uniformly integrable in $L^1(0,T; V_0^*)$, hence also by
the Dunford--Pettis theorem that
\[
(A_\lambda(u_\lambda))_\lambda \quad\text{is sequentially weakly compact in } L^1(0,T; V_0^*)\,.
\]
We deduce that there exists $\xi\in L^1(0,T; V_0^*)$ such that,
on a not relabelled subsequence,
\beq
  \label{conv4}
  A_\lambda(u_\lambda) \wto \xi \qquad\text{in } L^1(0,T; V_0^*)\,.
\eeq
Furthermore, by the weak lower semicontinuity of the convex integrand
$\int_0^T\varphi_0^*(\cdot)\,ds$ and the estimate \eqref{est2}, we have 
\[
  \int_0^T\varphi_0^*(\xi(s))\,ds\leq 
  \liminf_{\lambda\searrow0}\int_0^T\varphi_0^*(A_\lambda(u_\lambda(s)))\,ds \leq M\,.
\]
Thanks to Lemma~\ref{lem:V0}, this implies that actually 
$\xi(t)\in E_\varphi^*$ for almost every $t\in(0,T)$ and that $\bar\varphi^*(\xi)\in L^1(0,T)$,
hence in particular that $\xi$ is essentially $L_{\bar\varphi^*}$-valued.
Moreover, proceeding as in the proof of Theorem~\ref{thm:chain}
we have that $L_{\bar\varphi^*}$ is a Borel subset of $V_0^*$,
hence 
$\xi\colon [0,T]\to L_{\bar\varphi^*}$ is Borel measurable and
$\norm{\xi}_{\bar\varphi^*}$ is measurable.
Since $\norm{y}_{\bar\varphi^*}\leq1+\bar\varphi^*(y)$ 
for all $y\in L_{\bar\varphi^*}$, we have that $\norm{\xi}_{\bar\varphi^*}\in L^1(0,T)$.
Furthermore, it also holds that $\xi\in L^1_w(0,T; L_{\bar\varphi^*})$.
Indeed, under {\bf H2ii} this is immediate since $L_{\bar\varphi^*}$
is separable and $\xi\in L^1(0,T; L_{\bar\varphi^*})$,
while under {\bf H2i} the weak measurability follows directly 
from the density of $V_0$ in $L_\varphi$ and the strong measurability 
of $\xi$ in $V_0^*$.

\subsection{Passage to the limit}
The approximated problem can be written as
\[
  u_\lambda(t) + \int_0^tA_\lambda(u_\lambda(s))\,ds = 
  u_0 + \int_0^tf(s)\,ds \qquad\forall\,t\in[0,T]\,.
\]
Fix now $t\in[0,T]$ arbitrary. By the convergence \eqref{conv4},
it follows that as $\lambda\searrow0$
\[
  \int_0^tA_\lambda(u_\lambda(s))\,ds \wto 
   \int_0^t\xi(s)\,ds \quad\text{in } V_0^*\,.
\]
By comparison, we deduce that $u_\lambda(t)$ converges weakly in $V_0^*$,
yielding thanks to \eqref{conv1} that
\[
  u_\lambda(t)\wto u(t) \quad\text{in } H\,.
\]
Hence, we have that 
\[
  u(t) + \int_0^t\xi(s)\,ds = 
  u_0 + \int_0^tf(s)\,ds \qquad\forall\,t\in[0,T]\,.
\]
This implies in particular also that $u\in W^{1,1}_w(0,T; L_{\bar\varphi^*})$,
with $\partial_t u =-\xi + f$,
hence also $u\in C^0([0,T]; H)$ by Theorem~\ref{thm:chain},
and $(u,\xi)$ solves \eqref{u}--\eqref{eq}.

We only need to show that $\xi$ is the weak realization of $\partial\varphi(u)$ in $L_{\bar\varphi^*}$,
namely condition \eqref{sub}.
To this end, note that we already proved that for every $t\in[0,T]$
\[
  \frac12\norm{u_\lambda(t)}_H^2
  +\int_0^t\left(A_\lambda(u_\lambda(s)), u_\lambda(s)\right)\,ds=
  \frac12\norm{u_0}_H^2
  +\int_0^t\left(f(s), u_\lambda(s)\right)\,ds\,,
\]
which yields, by the weak lower semicontinuity of the $H$-norm and the convergence \eqref{conv1},
\begin{align*}
  \limsup_{\lambda\searrow0}
  \int_0^T\left(A_\lambda(u_\lambda(s)), u_\lambda(s)\right)\,ds
  &=
  \frac12\norm{u_0}_H^2
  +\int_0^T\left(f(s), u(s)\right)\,ds
  -\frac12\liminf_{\lambda\searrow0}\norm{u_\lambda(T)}_H^2\\
  &\leq\frac12\norm{u_0}_H^2
  +\int_0^T\left(f(s), u(s)\right)\,ds-\frac12\norm{u(T)}_H^2\,.
\end{align*}
Furthermore, since we have proved that 
\[
  \partial_t u + \xi = f \quad\text{in } L_{\bar\varphi^*}\,, \qquad u(0)=u_0\,,
\]
taking the $[\cdot, \cdot]$ duality with $u$ and integrating on $(0,t)$,
using Theorem~\ref{thm:chain} we get exactly that 
\[
  \frac12\norm{u(t)}_H^2
  +\int_0^t\left[\xi(s), u(s)\right]\,ds=
  \frac12\norm{u_0}_H^2
  +\int_0^t\left(f(s), u(s)\right)\,ds \qquad\forall\,t\in[0,T]\,,
\]
which in particular proves the energy equality \eqref{energy}.
Putting this information together we obtain
\beq
  \label{limsup}
  \limsup_{\lambda\searrow0}\int_0^T\left(A_\lambda(u_\lambda(s)), u_\lambda(s)\right)\,ds\leq
  \int_0^T\left[\xi(s), u(s)\right]\,ds\,.
\eeq
Now, recalling that $A_\lambda\in A(J_\lambda(\cdot))$, we have that 
\[
  \varphi(J_\lambda(u_\lambda)) +
  \left(A_\lambda(u_\lambda),  z-J_\lambda(u_\lambda)\right) \leq
  \varphi(z) 
  \quad\text{a.e.~in } (0,T)\,,
  \qquad\forall\,z\in L^2(0,T; H)\,, 
\]
so in particular it holds that 
\[
  \int_0^T\varphi(J_\lambda(u_\lambda(s)))\,ds +
  \int_0^T\left(A_\lambda(u_\lambda(s)),  z(s)-J_\lambda(u_\lambda(s))\right)\,ds \leq
  \int_0^T\varphi(z(s))\,ds\,.
\]
Now, we want to let $\lambda\searrow0$ in the inequality.
To this end, note first that the convergence \eqref{conv3} and the weak
lower semicontinuity of the convex integrands yields 
\[
  \int_0^T\varphi(u(s))\,ds \leq\liminf_{\lambda\searrow 0}
  \int_0^T\varphi(J_\lambda(u_\lambda(s)))\,ds\,.
\]
Secondly, the weak convergence \eqref{conv4} readily implies that,
for all $z\in L^\infty(0,T; V_0)$,
\[
  \lim_{\lambda\searrow0}
  \int_0^T\left(A_\lambda(u_\lambda(s)),  z(s))\right)\,ds=
  \int_0^T\ip{\xi(s)}{z(s)}_{V_0^*, V_0}\,ds
  =\int_0^T[\xi(s),z(s)]\,ds\,.
\]
Finally, the limsup inequality \eqref{limsup} yields 
\begin{align*}
  &\limsup_{\lambda\searrow0}\int_0^T
  \left(A_\lambda(u_\lambda(s)),  J_\lambda(u_\lambda(s))\right)\,ds\\
  &=
  \limsup_{\lambda\searrow0}\int_0^T\left[
  \left(A_\lambda(u_\lambda(s)), u_\lambda(s)\right)
  -\lambda\norm{A_\lambda(u_\lambda(s))}_H^2
  \right]\,ds\\
  &\leq \limsup_{\lambda\searrow0}\int_0^T
  \left(A_\lambda(u_\lambda(s)), u_\lambda(s)\right)\,ds\leq
  \int_0^T[\xi(s), u(s)]\,ds\,.
\end{align*}
Hence, letting $\lambda\searrow0$ we infer that, for all 
\[
  \int_0^T\varphi(u(s))\,ds +
  \int_0^T\left[\xi(s),  z(s)-u(s)\right]\,ds \leq
  \int_0^T\varphi(z(s))\,ds \qquad\forall\,z\in L^\infty(0,T; V_0)\,.
\]
By a standard localization procedure and by the density of $V_0$ in $E_\varphi$
we have 
\[
  \varphi(u) +
  \left[\xi,  x-u\right] \leq
  \varphi(x) \qquad\forall\,x\in E_\varphi\,,\quad\text{a.e.~in } (0,T)\,.
\]
This complete the proof of condition \eqref{sub} and of
existence of solutions in Theorem~\ref{thm}.

\subsection{Continuous dependence}
We prove here the continuous dependence in Theorem~\ref{thm2}, 
which in particular implies uniqueness
of solutions. 

In the setting and notations of Theorem~\ref{thm2} we have that 
\[
  \partial_t(u_1-u_2) + \xi_1-\xi_2 = f_1-f_2 \quad\text{a.e.~in } (0,T)\,,\qquad
  (u_1-u_2)(0)=u_0^1-u_0^2\,.
\]
Moreover, note that by convexity and symmetry of $\varphi$ we have 
\[
  \varphi\left(\frac{u_1-u_2}2\right)\leq \frac12\varphi(u_1) + \frac12\varphi(u_2) \in L^1(0,T)\,,
\]
and similarly 
\[
  \bar\varphi^*\left(\frac{\partial_tu_1-\partial_t u_2}{2}\right)\leq
  \frac12\bar\varphi^*(\partial_tu_1) + \frac12\bar\varphi^*(\partial_tu_2) \in L^1(0,T)\,.
\]
Consequently, taking the $[\cdot,\cdot]$ duality with $u_1-u_2$, 
integrating on $(0,t)$ and using Theorem~\ref{thm:chain} we have,
for every $t\in[0,T]$, 
\begin{align*}
  &\frac12\norm{(u_1-u_2)(t)}_H^2 + \int_0^t\left[(\xi_1-\xi_2)(s), (u_1-u_2)(s)\right]\,ds\\
  &=\frac12\norm{u_0^1-u_0^2}_H^2 + \int_0^t\left((f_1-f_2)(s), (u_1-u_2)(s)\right)\,ds\,.
\end{align*}
Now, from \eqref{sub} we know that 
\[
  \varphi(u_1)+\bar\varphi^*(\xi_1)=[\xi_1,u_1]\,, \qquad
  \varphi(u_2)+\bar\varphi^*(\xi_2)=[\xi_2,u_2]\,,
\]
from which
\[
  [\xi_1-\xi_2, u_1-u_2]=\varphi(u_1)+\bar\varphi^*(\xi_1)
  +\varphi(u_2)+\bar\varphi^*(\xi_2) - [\xi_1, u_2] - [\xi_2, u_1]\,.
\]
By the Young inequality we also deduce that
\[
  [\xi_1, u_2]\leq \bar\varphi^*(\xi_1) + \varphi(u_2)\,, \qquad
  [\xi_2, u_1]\leq \bar\varphi^*(\xi_2) + \varphi(u_1)\,
\]
so that putting everything together we infer that 
\[
  [\xi_1-\xi_2, u_1-u_2]\geq0 \quad\text{a.e.~in } (0,T)\,.
\]
We deduce that 
\[
  \frac12\norm{(u_1-u_2)(t)}_H^2
  \leq\frac12\norm{u_0^1-u_0^2}_H^2 + \int_0^t\norm{(f_1-f_2)(s)}_H\norm{(u_1-u_2)(s)}_H\,ds
\]
for every $t\in[0,T]$, and the thesis follows then by the Gronwall lemma.


\section{Applications}
\label{sec:applications}
In this section, we thoroughly discuss a wide spectrum of applications. 
First of all, we show that the classical variational theory in reflexive and separable spaces is covered as a special case of our results.
Then, we show the applicability of our theory to much more general examples, such as evolution equations in singular Orlicz spaces, in Muselak--Orlicz spaces, and Musielak--Orlicz--Sobolev spaces.
These cover, among many other examples, PDEs in variable-exponent Sobolev spaces, in double-phase spaces, and PDE with dynamic boundary conditions driven by singular potentials.

\subsection{The classical variational theory}
\label{ssec:var}
We show here that the classical variational theory for evolution
equations is covered by our results: in this direction 
we refer the reader to the main contributions 
\cite[Thm.~4.10]{barbu-monot} and \cite{ak-ot, brez}.

Let $H$ be a Hilbert space and $V$ a separable reflexive Banach space 
such that $V\embed H$ continuously and densely.
Let $\varphi\colon V\to[0,+\infty)$ be convex, lower semicontinuous,
with $\varphi(0)=0$,
such that $V=D(\partial\varphi)$, 
where $D(\partial\varphi)$ denotes a domain of $\partial\varphi$,
and there exist 
constants $c_1$, $c_2>0$ and $p\geq2$ such that 
\[
  \ip{y}{x}\geq c_1\norm{x}_V^p\,, \quad
  \norm{y}^{p'}_{V^*}\leq c_2(1+\norm{x}_V^p)\,,
  \qquad\forall\,x\in V\,,\quad \forall\,y\in\partial\varphi(x)\,,
\]
where $p':=\frac{p}{p-1}$.
Let also $u_0\in H$
and $f\in L^2(0,T; H)$.
Then the classical variational theory ensures that there exists
a unique $(u,\xi)$ with 
\[
  u\in W^{1,p'}(0,T; V^*)\cap C^0([0,T]; H)\cap L^p(0,T; V)\,, \qquad
  \xi\in L^{p'}(0,T; V^*)
\]
such that 
\[
  \partial_t u + \xi = f\,, \qquad \xi\in\partial\varphi(u)\,, \qquad u(0)=u_0\,.
\]

Let us compare this with our results.
In this setting, it is not difficult to check that $\varphi$
(suitably extended to $+\infty$ on $H\setminus V$) is actually a modular,
and that assumption {\bf H0} holds with the choice $s=p$.
Furthermore, the growth conditions on $\partial\varphi$ imply that 
\[
  c_1'\norm{x}_V^p \leq \varphi(x) \leq c_2'(1 + \norm{x}_V^p) \qquad\forall\,x\in V\,,
\]
for some constants $c_1'$, $c_2'>0$. Consequently, we easily deduce that 
in this case
\[
  E_\varphi=L_\varphi= V\,,
\]
so that {\bf H1} holds with the trivial choice $V_0=V$.
Clearly, one has then $\bar\varphi=\varphi$, $E_{\varphi}^*=V^*$,
and $\bar\varphi^*=\varphi^*$. The growth conditions on $\varphi$ yields 
\[
  c_3\norm{y}_{V^*}^{p'} - 1/c_3'\leq
  \varphi^*(y) \leq c_4'(1 + \norm{y}_{V^*}^{p'}) \qquad\,y\in V^*\,,
\]
for some $c_3'$, $c_4'>0$, 
so that $E_{\bar\varphi^*}=L_{\bar\varphi^*}=E_{\varphi}^*=V^*$.
Finally, since $V$ is both reflexive and dense in $H$, 
it is a standard matter to check that both {\bf H2i} and {\bf H2ii} are satisfied.

Our main result Theorem~\ref{thm} implies that 
for all $u_0\in H$
and $f\in L^2(0,T; H)$,
there is a unique pair $(u,\xi)$ with 
\begin{align*}
  &u\in W^{1,1}_w(0,T; V^*)\cap C^0([0,T]; H)\cap L^1_w(0,T; V)\,, \qquad
  \xi\in L^1_w(0,T; V^*)\,,\\
  &\varphi(u),\bar\varphi^*(\xi)\in L^1(0,T)\,,
\end{align*}
such that 
\[
  \partial_t u + \xi = f\,, \qquad \xi\in\partial\varphi(u)\,, \qquad u(0)=u_0\,.
\]
Since $V$ and $V^*$ are separable, one actually has 
\[
  u\in W^{1,1}(0,T; V^*)\cap L^1(0,T; V)\,, \qquad \xi\in L^1(0,T; V^*)\,.
\]
Moreover, thanks to the growth conditions on $\varphi$, 
we immediately see that the regularity $\varphi(u)$, $\bar\varphi^*(\xi)\in L^1(0,T)$ yields 
\[
  u\in L^p(0,T; V)\,,\qquad \xi\in L^{p'}(0,T; V^*)\,.
\]
Finally, by comparison in the equation we have $\partial_t u\in L^{p'}(0,T; V^*)$ as well.

Hence, in this simplified setting, our Theorem~\ref{thm} actually coincides with 
the classical existence result from the variational theory.

\subsection{Reaction-diffusion equations}
\label{ssec:RD}
We present here a first simple example of a class of PDEs
that falls out of the classical variational setting presented in Subsection~\ref{ssec:var},
but that can nonetheless be covered by our main existence result.

Let us consider partial differential equations in the form 
\beq
  \label{RD}
  \begin{cases}
  \partial_t u -\operatorname{div}(\partial M(\nabla u)) + \partial N(u) \ni f
  \quad&\text{in } (0,T)\times\Omega\,,\\
  u=0 \quad&\text{on } (0,T)\times\partial\Omega\,,\\
  u(0)=0 \quad&\text{in } \Omega\,,
  \end{cases}
\eeq
where $\Omega\subset\erre^d$ is a bounded Lipschitz domain, 
$T>0$ is a fixed final time, 
the data are chosen as 
$u_0\in L^2(\Omega)$ and $f\in L^2((0,T)\times \Omega)$,
and
$M\colon \erre^d\to[0,+\infty)$ and $N\colon \erre\to\erre$ are 
convex, even, lower semicontinuous and polynomially 
bounded by above and below.
More specifically, we suppose that
there exist $c_1$, $c_2>0$
and $p$, $q\geq2$ such that 
\begin{align*}
  c_1|x|^p\leq y\cdot x\,, \quad |y|^{p'}\leq c_2(1+|x|^p)
  \qquad&\forall\,x\in\erre^d\,,\quad\forall\,y\in\partial M(x)\,,\\
  c_1|r|^q\leq wr\,, \quad |w|^{q'}\leq c_2(1+|r|^q) 
  \qquad&\forall\,r\in\erre\,,\quad\forall\,w\in\partial N(r)\,,
\end{align*}
where $p'$ and $q'$ are the conjugate exponents of $p$ and $q$, respectively.
Here the case 
$M(\nabla u) = |\nabla u|^p$
leads to thoroughly studied $p$-Laplace operator: see, for example, \cite{ladSolUra1968,LadUra1968}.

Since $p$ and $q$ may be different in general, 
this setting does not fall directly in the classical framework 
presented in Subsection~\ref{ssec:var}, as the coercivity 
condition is not satisfied if $p\neq q$.
However, let us show that it can be treated by 
using our results.

One can consider the modular $\varphi\colon L^2(\Omega)\to[0,+\infty]$ defined as
\[
  \varphi(u):=\begin{cases}\displaystyle
  \int_\Omega M(\nabla u) + \int_\Omega N(u) 
  \quad&\text{if } u\in W^{1,p}_0(\Omega)\cap L^q(\Omega)\,,\\
  +\infty \quad&\text{otherwise}\,.
  \end{cases}
\]
With this notation, the PDE \eqref{RD}
can be written in the abstract form 
\[
  \partial_t u + \partial\varphi(u)\ni f\,, \qquad u(0)=u_0\,,
\]
by choosing 
\[
H:=L^2(\Omega)\,,\qquad V:=W^{1,p}_0(\Omega)\cap L^q(\Omega)\,,\qquad
V^*=W^{-1,p'}(\Omega)+ L^{q'}(\Omega)\,.
\]
Thanks to the growth conditions on $M$ and $N$ we deduce that 
\[
  c_1'\left(\norm{\nabla u}_{L^p(\Omega)}^p + \norm{u}_{L^q(\Omega)}^q\right)
  \leq\varphi(u)\leq
  c_2'\left(1+\norm{\nabla u}_{L^p(\Omega)}^p + \norm{u}_{L^q(\Omega)}^q\right)
  \qquad\forall\,u\in V\,.
\]
Consequently, assumption {\bf H0} holds with $s=\min\{p,q\}$. Furthermore, 
the growth condition readily implies that 
$V=E_\varphi=L_\varphi$ and 
$E_{\bar\varphi^*}=L_{\bar\varphi^*}=E_{\varphi}^*=V^*$, 
so that also {\bf H1--H2i--H2ii} are satisfied, 
with the choice $V_0=V$. Moreover, 
the conjugate $\bar\varphi^*\colon V^*\to[0,+\infty)$ satisfies
\[
  \bar\varphi^*(v)=\ip{v}{u}-\varphi(u) \qquad\forall\,u\in V\,,\quad\forall\,v\in\partial\bar\varphi(u)\,,
\]
where by \cite[Thm.~2.10]{barbu-monot} we have that 
\[
  \partial\bar\varphi(u)=\left\{-\operatorname{div}v_1 + v_2: \quad v_1\in\partial M(\nabla u)\,,\;
  v_2\in\partial N(u) \quad\text{a.e.~in } \Omega\right\}
  \qquad \forall\,u\in V\,.
\]
In this expression, by assumption on $M$ and $N$ we have that
$v_1\in L^{p'}(\Omega)^d$ and $v_2\in L^{q'}(\Omega)$, and
the divergence is intended in the sense of distributions on $\Omega$.

The existence Theorem~\ref{thm} ensures then that
for all
$u_0\in L^2(\Omega)$ and 
$f\in L^2((0,T)\times \Omega)$ there exists a unique pair $(u,\xi)$
with 
\begin{align*}
  &u\in W^{1,1}_w(0,T; V^*)\cap C^0([0,T]; H)\cap L^1_w(0,T; V)\,, \qquad
  \xi\in L^1_w(0,T; V^*)\,,\\
  &\varphi(u),\, \bar\varphi^*(\xi)\in L^1(0,T)\,,
\end{align*}
such that 
\[
  \partial_t u + \xi = f\,, \qquad \xi\in\partial\bar\varphi(u)\,, \qquad u(0)=u_0\,.
\]
Since $V$ and $V^*$ are separable, these conditions imply 
the existence and uniqueness of a solution of \eqref{RD} with
\begin{align*}
  &u\in L^p(0,T; W^{1,p}_0(\Omega))\cap L^q(0,T; L^q(\Omega))\,,\\
  &\partial_t u\in L^{p'}(0,T; W^{-1,p'}(\Omega)) + L^{q'}(0,T; L^{q'}(\Omega))\,,\\ 
  &\xi=-\operatorname{div}\xi_1 + \xi_2\,, \qquad\xi_1\in L^{p'}(0,T; L^{p'}(\Omega)^d)\,,
  \quad\xi_2\in L^{q'}(0,T; L^{q'}(\Omega))\,,
\end{align*}
and
\[
  \partial_t u -\operatorname{div}\xi_1+\xi_2\ni f \quad\text{in } V^*\quad\text{a.e.~in } (0,T)\,,
  \qquad u(0)=u_0\,.
\]

\subsection{Singular PDEs in Musielak--Orlicz spaces}
\label{ssec:MO}
In this subsection, we examine our approach in the setting of Muselak--Orlicz spaces.
For all the abstract theory and general properties we refer
to the classical monograph \cite[\S~7]{mus}.
We will show that our results cover several interesting cases of singular PDEs,
including evolution equations in both reflexive and non-reflexive spaces, such as
Lebesgue spaces with variable exponents, 
double-phase spaces, Orlicz spaces, and weighted Lebesgue spaces. 

Given a bounded domain $\Omega\subset\erre^d$ regular enough, 
we consider $\varphi_M \colon L^2(\Omega)\to[0,+\infty]$ of the form
\[
  \varphi_M (v) = 
  \begin{cases}
  \displaystyle
  \int_\Omega M(x, v(x)) \, dx \quad&\text{if } M(\cdot, v) \in L^1(\Omega)\,,\\
  +\infty \quad&\text{otherwise}\,.
  \end{cases}
\]
where $M\colon \Omega \times \erre \to [0,\infty)$ is a generalized strong $\Phi$-function \cite{Orlicz19} in the sense that
\begin{enumerate}
  \item $M(\cdot,z)$ is measurable for every $z \in\erre$;
  \item $M(x,\cdot)$ is convex and continuous for almost every $x\in\Omega$;
  \item $M(x,0)= \lim_{z \to 0}M(x,z) = 0$ and 
  $\lim_{z \to \infty}M(x,z) = +\infty$ for almost all $x\in\Omega$;
  \item there is $\eps\in(0,1]$ such that $M(x,\eps)\leq1$
  and $M(x,1/\eps)\geq1$ for almost all $x\in\Omega$;
  \item if $M(x,\alpha z)=0$ for all $\alpha>0$ and almost all $x\in\Omega$, 
  then $z=0$;
  \item $M(x,z)=M(x,-z)$ for all $z\in\erre$ and almost all $x\in\Omega$.
\end{enumerate}In this setting, $\varphi_M$ is a lower semicontinuous convex semi-modular on $L^2(\Omega)$.
Let us use the classical notation $L^{M}(\Omega):=L_{\varphi_M}$ and 
$E^M(\Omega):=E_{\varphi_M}$
for the respective Muselak--Orlicz spaces.
Moreover, if there exists $c>0$ such that
\[
M(x,z) \geq cz^2 \qquad\text{for a.e.~}x\in\Omega\,,\quad\forall\,z\in\erre\,,
\]
then we have $L^M(\Omega) \subset L^2(\Omega)$. Hence, we can choose
$H:=L^2(\Omega)$, and assumption {\bf H0} holds with $s = 2$.

We denote by $M^*\colon\Omega\times\erre\to[0,+\infty]$ the convex conjugate of $M$
with respect to its second variable, namely
\[
  M^*(x,z):=\sup_{y\in\erre}\left\{zy - M(x,y)\right\}\,, \qquad(x,z)\in\Omega\times\erre\,.
\]

\begin{defin}
  A function $M\colon \Omega \times \erre\to[0,+\infty]$ 
  satisfies the weak doubling condition $\Delta^w_2$ 
  if there exist a constant $k \geq 2$ and $h \in L^1(\Omega)$ such that
  \begin{equation}\label{cond:delta_2}
  	M(x,2z) \leq k M(x,z) + h(x)
  \end{equation}
  for almost all $x \in \Omega$ and all $z \in\erre$. 
  Similarly, $M$ fulfils the condition $\nabla^w_2$ if $M^*$ satisfies $\Delta^w_2$.
  Whenever $h\equiv 0$, the conditions are referred to as
  strong $\Delta_2$ and strong $\nabla_2$.
\end{defin}
\begin{defin}
  Let $M, N\colon \Omega \times \erre\to[0,+\infty]$.
  We say that $M\preceq N$
  if there exist two constants $c_1$, $c_2>0$ and $h\in L^1(\Omega)$ such that 
  \[
  M(x,z) \leq c_1N(x,c_2z) +h(x) \qquad\text{for almost every $x\in\Omega$
  and for all $z\in\erre$}\,.
  \]
  We say that $M$ and $N$ are equivalent if $M\preceq N\preceq M$.
\end{defin}
It is well-known that every function $M$ meeting the 
$\Delta^w_2$ $(\nabla^w_2)$ condition
has an equivalent function $N$ which satisfies the strong $\Delta_2$ $(\nabla_2)$ condition.

\begin{defin}\label{loc_int}
  A generalized strong $\Phi$-function $M$ is said 
  locally integrable if, for every measurable compact subset $K\subset \Omega$
  and for every $z\in\erre$, it holds that
  \[
  \int_KM(x,z)\, dx<+\infty\,.
  \]
\end{defin}

We recall the main properties of the Musielak--Orlicz spaces in the following proposition:
for detailed proofs, the reader can refer to \cite{mus, youss, Chl2018, har}.
\begin{prop}
  \label{prop:properties}
  Let $M$ be a generalized strong $\Phi$-function.
  Then, the following holds:
  \begin{itemize}
    \item[(i)] $E^M(\Omega)$ and $L^M(\Omega)$ are Banach spaces w.r.t.\ $\|\cdot\|_{\varphi_M}$,
  and $L^\infty(\Omega)$ is continuously embedded in $E^M(\Omega)$;
  \item[(ii)] if $M$ is locally integrable, then 
  the simple functions on $\Omega$ and
  the smooth functions with compact support $C^\infty_c(\Omega)$ are dense
  in $E^M(\Omega)$ w.r.t.~the norm\ $\|\cdot\|_{\varphi_M}$.
  In particular, $E^M(\Omega)$ is separable.
 Moreover, we have the characterization of the dual 
 $E^M(\Omega)^*\cong L^{M^*}(\Omega)$;
  \item[(iii)] if $M$ satisfies the $\Delta_2$ condition, then $L^M(\Omega)=E^M(\Omega)$ and 
  there exists $p>1$ such that $\frac{M(x,z)}{|z|^p}\to 0$ as $|z|\to+\infty$ for 
  almost all $x\in \Omega$;
  \item[(iv)] $L^M(\Omega)$ is reflexive if and only if $M$ satisfies both 
$\Delta^w_2$ and $\nabla^w_2$ conditions.
  \end{itemize}
\end{prop}

The evolution equation associated to this choice of $\varphi_M$ reads
\beq
  \label{MO}
  \begin{cases}
  \partial_t u(t,x) + \partial M(x,u(t,x)) \ni f(t,x)
  \quad&(t,x)\in(0,T)\times\Omega\,,\\
  u(0,x)=u_0(x) \quad&x\in\Omega\,,
  \end{cases}
\eeq
where, as in the previous subsection, 
$T>0$ is a fixed final time, and
the data are chosen as 
$u_0\in L^2(\Omega)$ and $f\in L^2((0,T)\times \Omega)$.

We have already pointed out that in the setting above, 
assumption {\bf H0} is satisfied with the choices $H:=L^2(\Omega)$ and $s=2$.
We analyze now in detail the validity of 
the hypotheses {\bf H1}, {\bf H2i--\bf H2ii} 
in connection to 
the $\Delta_2$ and $\nabla_2$ conditions for $M$.

\underline{\bf $M$ satisfies both $\Delta_2$ and $\nabla_2$.}\\
If $M$ satisfies both the $\Delta_2$ and $\nabla_2$ conditions, 
then $E^M(\Omega)=L^M(\Omega)$ is reflexive.
Moreover, if also $M$ and $M^*$ are locally integrable on $\Omega$,
by property (ii) of Proposition~\ref{prop:properties},
$E^M(\Omega)$ is separable and $E^M(\Omega)^*=L^{M^*}(\Omega)$.
In particular, it follows that $E^M(\Omega)^*=L^{M^*}(\Omega)=E^{M^*}(\Omega)$.
Hence, assumption {\bf H1} is satisfied by the trivial choice $V_0=E^M(\Omega)$,
and assumptions {\bf H2i--H2ii} hold
since $E^M(\Omega)=L^M(\Omega)$ and $H$ is dense 
in $L^{M^*}(\Omega)=E^{M^*}(\Omega)$.

In this setting, our Theorem~\ref{thm} ensures then that 
the equation \eqref{MO} has a unique solution 
\[
  u \in W^{1,1}(0,T; L^{M^*}(\Omega))\cap C^0([0,T]; L^2(\Omega)) \cap L^1(0,T; L^M(\Omega))\,,
  \quad
  \xi\in L^1(0,T; L^{M^*}(\Omega))\,,
\]
such that 
\[ 
  M(\cdot, u),\, M^*(\cdot, \xi) \in L^1(0,T)\,.
\]
Let us stress that the separability properties of $L^M(\Omega)$ and $L^{M^*}(\Omega)$
ensure actually that the measurability in time of such solutions is intended in
the usual strong sense. Moreover, 
thanks to the characterisation of $\partial M$ in \cite[Prop.~2.7]{barbu-monot}
the subdifferential relation \eqref{sub}
in this case can be written pointwise and reads
\[
  \xi \in \partial M(\cdot, u) \quad\text{a.e.~in } (0,T)\times\Omega\,.
\]

This framework where $M$ is both $\Delta_2$ and $\nabla_2$
allows to cover several interesting cases, coming 
mainly from the modelling of anisotropic/non-homogenous phenomena.
We refer the interested reader to the surveys 
\cite{HarHasLeNuo2010} and \cite{Chl2018} for more details.
For instance, the following well-known examples 
are included in this setting:
\begin{itemize}
  \item Variable exponents spaces:
         $M(x,z) = |z|^{p(x)}$, where $p\colon\Omega\to(1,+\infty)$ is measurable
         and such that 
         \[
         1<p^-:=\essinf_{x\in\Omega}p(x) \leq 
         \esssup_{x\in\Omega}p(x)=:p^+<+\infty\,.
	\]
 \item Double phase spaces: $M(x,z) = |z|^{p(x)} + a(x)|z|^{q(x)}$, where
 $a\colon\Omega\to[0,+\infty)$ is measurable and bounded, and
 \[
 1<p^-\leq p < q \leq q^+ <+\infty\,.
 \]
  \item Orlicz spaces: $M(x,z) = \Phi(z)$, independently of $x\in\Omega$,
  where $\Phi$ is $\Delta_2$ and $\nabla_2$.
  \item Weighted Lebesgue spaces.
\end{itemize}

\underline{\bf $M$ satisfies either $\Delta_2$ or $\nabla_2$.}\\
Let us start with $M$ fulfilling the $\Delta_2$-condition
(but not necessarily the $\nabla_2$ condition).
In this case, $E^M(\Omega)$ coincides with $L^M(\Omega)$. 
Moreover, if $M$ is locally integrable on $\Omega$,
by property (ii) of Proposition~\ref{prop:properties}
the space $E^M(\Omega)$ is separable and dense in $L^2(\Omega)$.
Furthermore, by property (iii) there exists $p>1$ such that
$\lim_{|z| \to +\infty}{M(x,z)}/{|z|^p} = 0$ for a.e.~$x \in \Omega$:
hence, if we choose $V_0$ as the space $L^p(\Omega)$,
then $L^p(\Omega)$ is dense $E^M(\Omega)$, 
separable and reflexive, and $\varphi_M$ is bounded on bounded subsets of $L^p(\Omega)$.
This shows that {\bf H1} is satisfied.
Furthermore, assumption {\bf H2} is trivially satisfied 
since {\bf H2i} holds as $E^M(\Omega)=L^M(\Omega)$.

Let us consider now $M$ fulfilling the $\nabla_2$-condition
(but not necessarily the $\Delta_2$ condition).
In this case, $E^{M^*}(\Omega)=L^{M^*}(\Omega)$. 
Moreover, if $M$ is locally integrable, then as before we have that 
$E^M(\Omega)$ is separable and dense in $L^2(\Omega)$.
Using also property (ii) of Proposition~\ref{prop:properties}, 
we have that $C^\infty_c(\Omega)$ is dense in $E^M(\Omega)$
and $L^\infty(\Omega)\embed E^M(\Omega)$ continuously.
Hence, supposing that $\Omega$ is regular enough
(for example, bounded with Lipschitz-boundary), by the 
classical Sobolev embeddings one can choose $V_0:=W^{m,p}(\Omega)$
with $m>\frac{d}{p}$, so that $C^\infty_c(\Omega)
\subset W^{m,p}(\Omega)\embed L^\infty(\Omega)$.
This shows that assumption {\bf H1} holds.
Furthermore, assumption {\bf H2} holds due to
to the following argument:
$M^*$ is $\Delta_2$, hence $L^{M^*}(\Omega)=E^{M^*}(\Omega)$ is separable,
and the collection of simple functions is a countable and dense subset of $L^{M^*}(\Omega)$.
Consequently, $L^2(\Omega)$ is dense in $L^{M^*}(\Omega)$, 
so that in particular {\bf H2ii} is satisfied.

In this setting, Theorem~\ref{thm} ensures then that 
the equation \eqref{MO} has a unique solution 
\[
  u \in W^{1,1}_w(0,T; L^{M^*}(\Omega))
  \cap C^0([0,T]; L^2(\Omega)) \cap L^1_w(0,T; L^M(\Omega))\,,
  \quad
  \xi\in L^1_w(0,T; L^{M^*}(\Omega))\,,
\]
such that 
\[ 
  M(\cdot, u),\, M^*(\cdot, \xi) \in L^1(0,T)\,.
\]
Let us point out that the measurability in time is not 
intended in the strong sense. Nonetheless, 
if $M$ satisfies the $\Delta_2$ condition, then 
$L^{M}(\Omega)$ is separable and actually it holds that also
$u\in L^1(0,T; L^M(\Omega))$.
Similarly, if $M$ satisfies the $\nabla_2$ conditions, then 
$L^{M^*}(\Omega)$ is separable and it holds also that
$u \in W^{1,1}(0,T; L^{M^*}(\Omega))$ and 
$\xi\in L^1(0,T; L^{M^*}(\Omega))$.
As before, the differential inclusion \eqref{sub}
can be written pointwise in~$\Omega$.

This more general framework where $M$
is allowed to satisfy either 
the $\Delta_2$ or the $\nabla_2$ condition
(but not necessarily both of them)
allows to cover almost all relevant cases of 
PDEs in Musielak--Orlicz spaces.
For these reasons, the variational theory 
presented here is widely applicable to 
most interesting examples.
For instance, let us mention the following \cite{Chl2018}:
\begin{itemize}
  \item Orlicz spaces: $M(x,z) = \Phi(z)$, independently of $x\in\Omega$,
  where $\Phi$ is either $\Delta_2$ or $\nabla_2$. For instance, 
  $\Phi(z)=\exp{|z|^p}-1$, $p\geq1$, or $\Phi(z)=(|z|+1)\ln(|z|+1) - |z|$.
\end{itemize}

\subsection{Singular PDEs in Musielak--Orlicz--Sobolev spaces}
\label{ssec:MOS}
In this subsection, we show that our results conver also 
evolution problems in Muselak--Orlicz--Sobolev spaces.
Again, for the abstract theory we refer
to the classical monograph \cite[\S~10]{mus}.
This framework includes interesting cases such as
singular or degenerate evolution equations in 
Sobolev spaces with variable exponents, double-phase spaces,
Orlicz--Sobolev spaces, and weighted Sobolev spaces.
For simplicity, we will only focus on homogeneous Dirichlet boundary conditions:
other classical choices can be easily covered with natural adaptation of this approach.

Here, we assume that $N$ and $M$ are generalized strong $\Phi$-functions
as in Subsection~\ref{ssec:MO}. 
Given a bounded domain $\Omega\subset\erre^d$ regular enough, 
we recall the definitions of the following Muselak--Orlicz--Sobolev spaces
(see \cite{Chl2018}):
\begin{align*}
  W^1_0L^M(\Omega)&:=\left\{v\in W^{1,1}_0(\Omega):\;
  v, |\nabla v|\in L^M(\Omega)\right\}\,,\\
  W^1_0E^M(\Omega)&:=\left\{v\in W^{1,1}_0(\Omega):\;
  v, |\nabla v|\in E^M(\Omega)\right\}\,,\\
  V^1_0L^M(\Omega)&:=\left\{v\in W^{1,1}_0(\Omega):\;
  |\nabla v|\in L^M(\Omega)\right\}\,,\\
  V^1_0E^M(\Omega)&:=\left\{v\in W^{1,1}_0(\Omega):\;
  |\nabla v|\in E^M(\Omega)\right\}\,.
\end{align*}
If $M$ is also locally integrable in the sense of Definition~\ref{loc_int}
then all the above spaces are actually Banach spaces.
Moreover, the space of compactly supported smooth functions 
$C^\infty_c(\Omega)$ is dense in $W^1_0E^M(\Omega)$
and in $V^1_0E^M(\Omega)$: a possible proof can be
readily adapted to the arguments of \cite[\S~2]{don-trud}. 
Also, if $M$ is $\Delta_2$ and $\nabla_2$ then 
$W^1_0L^M(\Omega)=W^1_0E^M(\Omega)$ is reflexive
(see \cite[Thm.~6.3]{har}).

We introduce the 
lower semicontinuous convex semi-modular
$\varphi_{M,N} \colon L^2(\Omega)\to[0,+\infty]$ as
\[
  \varphi_{M,N} (v) = 
  \begin{cases}
  \displaystyle
  \int_\Omega N(x, v(x)) \, dx + \int_\Omega M(x, |\nabla v(x)|) \, dx\\
  \qquad\text{if } N(\cdot, v)\in L^1(\Omega)\,, \quad v\in W^{1,1}_0(\Omega)\,,\quad
  M(\cdot, |\nabla v|) \in L^1(\Omega)\,,\\
  +\infty\\
   \qquad\text{otherwise}\,.
  \end{cases}
\]
Supposing again that, for some $c>0$,
\[
M(x,z), N(x,z) \geq cz^2 \qquad\text{for a.e.~}x\in\Omega\,,\quad\forall\,z\in\erre\,,
\]
we have that $L^M(\Omega), L^N(\Omega) \subset L^2(\Omega)$ and
that $W^1_0L^M(\Omega), V^1_0L^M(\Omega) \subset H^1_0(\Omega)$.
Hence, one can characterize the modular spaces associated to $\varphi_{M,N}$
in terms of the Musielak--Orlicz--Sobolev spaces related to $N$ and $M$ as
\begin{align*}
  L_{\varphi_{M,N}}&=\left\{v\in L^N(\Omega)\cap H^1_0(\Omega):
  \;|\nabla v|\in L^M(\Omega)\right\}= L^N(\Omega)\cap V^{1}_0L^M(\Omega)\,,\\
  E_{\varphi_{M,N}}&=\left\{v\in E^N(\Omega)\cap H^1_0(\Omega):
  \;|\nabla v|\in E^M(\Omega)\right\}= E^N(\Omega)\cap V^{1}_0E^M(\Omega)\,.
\end{align*}
As before, we can choose then
$H:=L^2(\Omega)$ and assumption {\bf H0} holds with $s = 2$.

Now, we set
${\boldsymbol M}:\Omega\times\erre^d\to\erre$ as
${\boldsymbol M}(x,z):=M(x, |z|)$, for $(x,z)\in\Omega\times\erre^d$,
so in particular ${\boldsymbol M}$ is symmetric in the second argument.
With this notation, the evolution equation associated to this choice of $\varphi_{M,N}$ reads
\beq
  \label{SMO}
  \begin{cases}
  \partial_t u(t,x) -\operatorname{div} \partial {\boldsymbol M}
  (x,\nabla u(t,x)) +  \partial N(x, u(t,x)) \ni f(t,x)
  \quad&(t,x)\in(0,T)\times\Omega\,,\\
  u(t,y)=0 \quad&(t,y)\in(0,T)\times\partial\Omega\,,\\
  u(0,x)=u_0(x) \quad&x\in\Omega\,,
  \end{cases}
\eeq
where again $T>0$ is a fixed final time,
$u_0\in L^2(\Omega)$, and $f\in L^2((0,T)\times \Omega)$.

Let us now discuss the validity of 
the hypotheses {\bf H1}, {\bf H2i--\bf H2ii}.
These strongly depend on whether $M$ and/or $N$
satisfy the conditions $\Delta_2$ and/or $\nabla_2$:
for sake of brevity, we only focus on two cases, the other
ones being analogous.

\underline{\bf $M$ and $N$ satisfy both $\Delta_2$ and $\nabla_2$.}\\
By the properties above, we have that the space
$E^N(\Omega)\cap V^1_0E^M(\Omega)
=L^N(\Omega)\cap V^1_0L^M(\Omega)$ is reflexive, 
and separable if also 
$M$ and $N$ are locally integrable.
In particular, it is immediate to check that ${\bf H1}$ holds 
with the trivial choice $V_0:=E^N(\Omega)\cap V^1_0E^M(\Omega)$,
as well as both assumptions {\bf H2i--H2ii}.

Theorem~\ref{thm} ensures then that 
the equation \eqref{SMO} has a unique solution 
\[
  u \in W^{1,1}(0,T; E_{\varphi_{M,N}}^*)
  \cap C^0([0,T]; L^2(\Omega)) 
  \cap L^1(0,T; E_{\varphi_{M,N}})\,,
  \quad
  \xi\in L^1(0,T; E_{\varphi_{M,N}}^*)\,,
\]
such that 
\[ 
  \varphi_{M,N}(u),\, \varphi_{M,N}^*(\xi) \in L^1(0,T)\,.
\]
Let us comment now on the subdifferential relation \eqref{sub}.
Let us notice that since $E_{\varphi_{M,N}} = 
E^N(\Omega)\cap V^1_0E^M(\Omega)$ and 
$E_{\varphi_{M,N}}$ is dense in both 
$E^N(\Omega)$ and $V^1_0E^M(\Omega)$,
we have the representation of the dual as
\[
  E_{\varphi_{M,N}}^* \cong
  L^{N^*}(\Omega) + V^1_0E^M(\Omega)^*\,.
\]
Let us show now that
the subdifferential relation \eqref{sub}
in this case can be written as
\[
  \xi=\xi_1 -\operatorname{div} \xi_2\,, \qquad
  \xi_1 \in \partial N(\cdot, u)\,, \quad
  \xi_2\in\partial {\boldsymbol M}(\cdot, \nabla u)
  \quad \text{a.e.~in } (0,T)\times\Omega\,.
\]
To this end, we notice that we have the representation 
$\varphi_{M,N}=\varphi_N + \psi_M$, where $\varphi_N$ is 
defined as in Subsection~\ref{ssec:MO} with respect to $N$, and 
$\psi_M\colon L^2(\Omega)\to[0,+\infty]$ is given by
\[
  \psi_M(v):=
  \begin{cases}
  \displaystyle
  \int_\Omega M(x, |\nabla v(x)|) \, dx
  \qquad&\text{if }  v\in H^{1}_0(\Omega)\,,\quad
  M(\cdot, |\nabla v|) \in L^1(\Omega)\,,\\
  +\infty
   \qquad&\text{otherwise}\,.
  \end{cases}
\]

\begin{lem}
  In this setting, if $M$ satisfies the $\Delta_2$ and $\nabla_2$ conditions, 
  the subdifferential of the restriction $\psi\colon V^1_0E^M(\Omega)\to \erre$ is the operator 
  \begin{align*}
  &A_M\colon V^1_0E^M(\Omega)\to 2^{V^1_0E^M(\Omega)^*}\,,\\
  &A_M(v):=\left\{-\operatorname{div}\eta:\quad
  \eta\in L^{M^*}(\Omega)^d\,,\quad\eta\in\partial {\boldsymbol M}
  (\cdot,\nabla v) \text{ a.e.~in } \Omega\right\}\,.
  \end{align*}
\end{lem}
\begin{proof}
  The proof can be directly adapted to the arguments of \cite[Thm.~2.17--2.18]{barbu-monot},
  by taking into account that under these assumptions the space $V^1_0E^M(\Omega)$
  is separable, reflexive, and dense in $L^2(\Omega)$.
\end{proof}
\begin{lem}\label{lem:sub}
  In this setting, if $M$ and $N$ satisfy the $\Delta_2$ and $\nabla_2$ conditions, 
  the subdifferential of the restriction 
  $\varphi_{M,N}\colon E^N(\Omega)\cap V^1_0E^M(\Omega)\to \erre$ is the operator 
  \begin{align*}
  &A_{M,N}:E^N(\Omega)\cap V^1_0E^M(\Omega)\to 
  2^{L^{N^*}(\Omega)+ V^1_0E^M(\Omega)^*}\,,\\
  &A_{M,N}(v):=\Big\{-\operatorname{div}\eta + \xi:\quad
  \eta\in L^{M^*}(\Omega)^d\,,\quad \xi \in L^{N^*}(\Omega)\,, \\
  &\hspace{12.7em}
  \eta\in\partial {\boldsymbol M}
  (\cdot,\nabla v)\,, \quad\xi\in\partial N(\cdot, v)\quad \text{ a.e.~in } \Omega\Big\}\,.
  \end{align*}
\end{lem}
\begin{proof}
  It follows from the classical result \cite[Thm.~2.11, Rmk.~2.1]{barbu-monot}.
\end{proof}

\underline{\bf $M$ and $N$ satisfy $\Delta_2$.}\\
If $M$ and $N$ fulfil the $\Delta_2$-condition
(but not necessarily the $\nabla_2$ condition),
then we have that $V^1_0E^M(\Omega)=V^1_0L^M\Omega)$
and $E^N(\Omega)=L^N(\Omega)$. Moreover, these spaces are
separable and dense in $L^2(\Omega)$
if also $M$ and $N$ are locally integrable.
As far as the choice of $V_0$ is concerned, 
combining the considerations made in Subsection~\ref{ssec:MO},
we can choose for example $V_0:=W^{1,p}_0(\Omega)$, where
$p$ realises condition (iii) of Proposition~\ref{prop:properties} for $M$ and $N$.
Clearly, $V_0$ is separable reflexive and dense in $E_{\varphi_{M,N}}$.
Furthermore, assumption {\bf H2} is trivially satisfied 
since {\bf H2i} holds as $E_{\varphi_{M,N}}=L_{\varphi_{M,N}}$.

Theorem~\ref{thm} ensures then that 
the equation \eqref{SMO} has a unique solution 
\[
  u \in W^{1,1}_w(0,T; E_{\varphi_{M,N}}^*)
  \cap C^0([0,T]; L^2(\Omega)) \cap L^1(0,T; L_{\varphi_{M,N}})\,,
  \quad
  \xi\in L^1_w(0,T; L^{M^*}(\Omega))\,,
\]
such that 
\[ 
  \varphi_{M,N}(u),\, \varphi_{M,N}^*(\xi) \in L^1(0,T)
\]
and the differential inclusion \eqref{sub} is satisfied.

\subsection{Singular PDEs with dynamic boundary conditions}
In this subsection we show that the variational theory presented in this 
papers also allows to consider PDEs with possibly singular dynamic boundary 
conditions. As a motivating example, let us focus now 
on problems in the following form:
\beq
  \label{eq:bound}
  \begin{cases}
  \partial_t u - \Delta u + \partial M(\cdot,u) \ni f
  \quad&\text{in } (0,T)\times\Omega\,,\\
  u=u_\Gamma \quad&\text{in } (0,T)\times\Gamma\,,\\
  \partial_t u_\Gamma +\partial_{\bf n}u
  +\partial M_\Gamma(\cdot,u_\Gamma) \ni f_\Gamma
  \quad&\text{in } (0,T)\times\Gamma\,,\\
  (u, u_\Gamma)(0)=(u_0, u_{0,\Gamma})
  \quad&\text{in } \Omega\times\Gamma\,,
  \end{cases}
\eeq
where $\Omega\subset\erre^d$ ($d\geq2$) is a  bounded domain
with sufficiently regular boundary $\Gamma$, 
$\Delta$
denotes the Laplace operator,
and $\bf n$ is the outward unit normal vector on $\Gamma$.
Here, $f\in L^2(0,T; L^2(\Omega))$ 
and $f_\Gamma\in L^2(0,T; L^2(\Gamma))$ represent two given forcing terms 
in the bulk and on the boundary, respectively, while 
$u_0\in L^2(\Omega)$ and $u_{0,\Gamma}\in L^2(\Gamma)$ are the given initial data.
Moreover, $M$ and $M_\Gamma$ are taken as
strong $\Phi$-fucntions on $\Omega$ and $\Gamma$, respectively, in the sense
of conditions (1)--(5) of
Subsection~\ref{ssec:MO}.

In order to frame the evolution problem \eqref{eq:bound}
in the context of modular spaces, it is natural
to consider the Hilbert space $H:=L^2(\Omega)\times L^2(\Gamma)$ and 
define $\varphi\colon H\to[0,+\infty]$ as
\begin{align*}
  \varphi(v,v_\Gamma):=
  \begin{cases}
  \displaystyle
  \int_\Omega\left(\frac12|\nabla v(x)|^2 + 
  M(x, v(x)) \right) dx +
  \int_\Gamma
  M_\Gamma(y, v_\Gamma(y)) dy,\\
  \qquad\text{if } v\in H^1(\Omega)\,, \quad M(\cdot, v) \in L^1(\Omega)\,, \\
  \qquad v_{|_\Gamma}=v_\Gamma\,, \quad
  M_\Gamma(\cdot, v_\Gamma) \in L^1(\Gamma)\,,\\
  +\infty\\
  \qquad\text{otherwise}\,.
  \end{cases}
\end{align*}
It is not difficult to check that problem \eqref{eq:bound}
can be formulated in an abstract way as
\[
  \partial_t {\bf u}  + \partial\varphi({\bf u}) \ni {\bf f}\,, \qquad
  \bu(0)=\bu_0\,,
\]
where we have used the bold notation to denote 
a general element of $H$, namely $\bu:=(u,u_\Gamma)$, 
${\bf f}:=(f,f_\Gamma)\in L^2(0,T; H)$ and $\bu_0:=(u_0, u_{0,\Gamma})\in H$.
Let us point out that for the general element ${\bf v}=(v,v_\Gamma)\in H$ 
it is not necessary true that $v_\Gamma$ is the trace of $v$ on $\Gamma$.
For this to be ensured, one needs on $v$ more regularity than just $L^2$:
for example, if ${\bf v}=(v,v_\Gamma)\in D(\varphi)$, then $v\in H^1(\Omega)$
and $v_\Gamma$ is its trace (which is indeed well-defined in $H^{1/2}(\Gamma)$).

Let us check that this problem can be framed in our variational setting.
First of all, using the notation of Subsection~\ref{ssec:MO}
for the Musielak--Orlicz spaces on $\Omega$ and $\Gamma$, we have that 
\begin{align*}
  L_\varphi&=\left\{{\bf v}\in H^1(\Omega)\times H^{1/2}(\Gamma):\quad
  v_{|_\Gamma}=v_\Gamma\,, \; v\in L^M(\Omega)\,,
  \; v_\Gamma\in L^{M_\Gamma}(\Gamma)\right\}\,,\\
  E_\varphi&=\left\{{\bf v}\in H^1(\Omega)\times H^{1/2}(\Gamma):\quad
  v_{|_\Gamma}=v_\Gamma\,,\; v\in E^M(\Omega)\,,
  \; v_\Gamma\in E^{M_\Gamma}(\Gamma)\right\}\,,
\end{align*}
and the $\norm{\cdot}_\varphi$-norm is equivalent to the 
norm of the intersection of the spaces appearing on the right-hand side, namely
\[
  \frac14\norm{{\bf v}}_\varphi \leq \norm{v}_{H^1(\Omega)} + \norm{v}_{\varphi_M}
  + \norm{v_\Gamma}_{\varphi_{M_\Gamma}}
   \leq4\norm{{\bf v}}_\varphi 
  \qquad\forall\,{\bf v}\in L_\varphi\,.
\]
It is immediate that $\varphi$ is a lower semicontinuous convex semi-modular 
on $H$, which is indeed a separable Hilbert space. Besides, the $s$-coercivity 
of $\varphi$ in the sense of assumption {\bf H0} follows from 
the respective hypothesis on $M$ and $M_\Gamma$
as in Subsection~\ref{ssec:MO} with $s=2$
(actually, in this case it is enough to require the above-mentioned coercivity 
only on $M_\Gamma$, thanks to a suitable Poincar\'e-type inequality and
the gradient contribution in $\varphi$).

Secondly, let us check assumption {\bf H1}. Using property (ii)
of Proposition~\ref{prop:properties}, we have that 
$L^\infty(\Omega)\times L^\infty(\Gamma)\embed 
E^M(\Omega)\times E^{M_\Gamma}(\Gamma)$
continuously, and $C^\infty(\overline\Omega)\times C^\infty(\Gamma)$
is dense in $E^M(\Omega)\times E^{M_\Gamma}(\Gamma)$.
Consequently,
a natural candidate for the space $V_0$ is
\[
  V_0:=\left\{{\bf v}\in H^{m}(\Omega)\times H^{m-1/2}(\Omega):
  \quad v_{|_\Gamma}=v_\Gamma\right\}\,,
  \qquad m>\frac{d}{2}\,.
\]  
Indeed, 
clearly $V_0$ is separable, reflexive, and
by the Sobolev embeddings we have the continuous inclusions
\[
  H^{m}(\Omega)\embed L^\infty(\Omega)\cap H^1(\Gamma)\,, \qquad 
  H^{m-1/2}(\Gamma)\embed L^\infty(\Gamma)\cap H^{1/2}(\Gamma)\,.
\]
Recalling the equivalence of the $\norm{\cdot}_\varphi$-norm above, 
this ensures that $V_0\embed E_\varphi$ continuously.
As far as the density of $V_0$ in $E_\varphi$ is concerned, given 
${\bf v}\in E_\varphi$, since $H^{m}(\Omega)$ is dense in 
$E^{M}(\Omega)\cap H^1(\Omega)$, there is a sequence 
$(v_{n})_n\subset H^{m}(\Omega)$ such that 
$v_{n}\to v_\Gamma$ in $E^{M}(\Omega)\cap H^1(\Omega)$.
Clearly, the traces $(v_{\Gamma,n})_n\subset H^{m-1/2}(\Gamma)$
satisfy $v_{\Gamma,n}\to v_\Gamma$ in $H^{1/2}(\Gamma)$.
Furthermore, if $d=2$, then $\Gamma$ has dimension $1$ and
by the Sobolev embeddings we have 
$H^{1/2}(\Gamma)\embed L^\infty(\Gamma)\embed E^{M_\Gamma}(\Gamma)$
continuously, so that also $v_{\Gamma,n}\to v_\Gamma$ in $E^{M_\Gamma}(\Gamma)$,
hence ${\bf v}_n\to {\bf v}$ in $E_\varphi$.
If $d=3$, then $\Gamma$ has dimension $2$ and by the Sobolev embeddings
we have $H^{1/2}(\Gamma)\embed L^4(\Gamma)$: in this case, 
if we suppose that $M_\Gamma$ is controlled by a $4$-th power, 
then $L^4(\Gamma)\embed E^{M_\Gamma}(\Gamma)$ and we can conclude as above.

Lastly, let us discuss assumption {\bf H2}.
Following the same line of Subsection~\ref{ssec:MO}
and without going into the details,
the main idea is that we can either assume $\Delta_2$-type conditions
in order to get {\bf H2i} or $\nabla_2$-type conditions in order to get {\bf H2ii}.
For example, if both $M$ and $M_\Gamma$ satisfy the $\Delta_2$ condition, 
then we have $E^M(\Omega)=L^M(\Omega)$ and 
$E^{M_\Gamma}(\Gamma)= L^{M_\Gamma}(\Gamma)$, so that condition {\bf H2i}
is trivially satisfied.

Theorem~\ref{thm} ensures then that the problem \eqref{eq:bound}
admits a unique variational solution
\[
  \bu \in W^{1,1}_w(0,T; L_{\bar\varphi^*})
  \cap C^0([0,T]; H) \cap L^1_w(0,T; L_\varphi)\,,
  \qquad
  {\boldsymbol\xi}\in L^1_w(0,T; L_{\bar\varphi^*})\,,
\]
such that 
\[ 
  \varphi(\bu),\, \varphi^*({\boldsymbol\xi}) \in L^1(0,T)\,,
\]
in the sense that 
\begin{align*}
  \int_\Omega u(t) \zeta &+ \int_\Gamma u_\Gamma(t) \zeta_\Gamma +
  \int_0^t\int_\Omega\nabla u\cdot\nabla\zeta +
  \int_0^t[\xi, \zeta]_{\varphi_M} 
  + \int_0^t[\xi_\Gamma, \zeta_\Gamma]_{\varphi_{M_\Gamma}}\\
  &=\int_\Omega u_0 \zeta + \int_\Gamma u_{0,\Gamma} \zeta_\Gamma
  +\int_0^t\int_\Omega f\zeta
  +\int_0^t\int_\Gamma f_\Gamma\zeta_\Gamma
  \qquad\forall\,t\in[0,T]\,,\qquad\forall\,{\boldsymbol\zeta}\in E_\varphi\,.
\end{align*}


\section*{Acknowledgement}
The authors are very grateful to the anonymous referees for the
valuable suggestions and remarks that improved the 
presentation of the results.
A.~Molchanova and L.~Scarpa gratefully acknowledge financial support
by the Austrian Science Fund~(FWF) through the projects M~2670 and M~2876, respectively.
A.~Molchanova has recieved funding from the European Union's Horizon 2020 research and innovation programme under the Marie Sk{\l}adowska-Curie grant agreement No 847693 and from the University of Vienna.
A.~Menovschikov was supported by the grant GA\v{C}R 20-19018Y.


\bibliographystyle{abbrv}
\def\cprime{$'$}

\end{document}